\documentclass[a4paper,11pt]{article}
\usepackage[utf8]{inputenc}
\usepackage[T1]{fontenc}
\usepackage{lmodern}
\usepackage{appendix}
\usepackage{amsmath}
\usepackage{amsfonts}
\usepackage{amssymb}
\usepackage[ruled, algosection,linesnumbered]{algorithm2e}
\usepackage[english]{babel}
\usepackage{graphicx}
\usepackage{setspace}
\usepackage{amsthm}
\usepackage[naturalnames=false]{hyperref}
\usepackage{authoraftertitle}
\usepackage{thmtools}
\usepackage{enumitem}
\usepackage{xcolor}
\usepackage{xspace}
\usepackage{stackrel}
\usepackage[right=4cm,left=3cm]{geometry}

\declaretheorem[numberwithin = section, name = Theorem]{theorem}
\declaretheorem[numbered = no, name = Theorem]{theorem*}
\declaretheorem[sibling = theorem, name = Corollary]{corollary}
\declaretheorem[sibling = theorem, name = Lemma]{lemma}
\declaretheorem[sibling = theorem, name = Remark, style=remark]{remark}

\declaretheorem[sibling = theorem, name = Definition, style = definition]{definition}

\newcommand{\aas}{\emph{a.a.s.}\xspace}

\newcommand{\abs}[1]{\left|#1\right|}
\newcommand{\Gnp}[2]{\mathcal{G}_{#1,#2}}
\newcommand{\R}{\mathbb{R}}
\newcommand{\N}{\mathbb{N}}
\newcommand{\Q}{\mathbb{Q}}
\newcommand{\Eof}[1]{\mathbb{E}\left[#1\right]}
\newcommand{\Pof}[1]{\mathbb{P}\left(#1\right)}

\newcommand{\chromaticOf}[1]{\chi \left(#1\right)}
\newcommand{\ceilOf}[1]{\left\lceil #1 \right\rceil}

\newcommand{\oOf}[1]{\textnormal{o}\left(#1\right)}

\newcommand{\Call}{\mathcal{C}}
\newcommand{\Clarge}[1]{\mathcal{L}^{#1}}
\newcommand{\Csmall}[1]{\mathcal{S}^{#1}}
\newcommand{\pert}[2]{\textnormal{pert}_{#1,#2}}
\newcommand{\pushright}[1]{\ifmeasuring@#1\else\omit\hfill$\displaystyle#1$\fi\ignorespaces}

\renewcommand{\epsilon}{\varepsilon}

\usepackage{fancyhdr}
\fancypagestyle{plain}{%
	\fancyhf{}

	\fancyfoot[R]{\thepage}
}

\fancypagestyle{appendix}{%
	\fancyhf{}
	\fancyhead[R]{\sectionmark}
	\setlength{\headheight}{28pt}
	
	\fancyfoot[R]{\thepage}
}

\setenumerate[0]{label=(\roman*)}

\makeatletter
\def\@seccntformat#1{\@ifundefined{#1@cntformat}%
	{\csname the#1\endcsname\quad}  
	{\csname #1@cntformat\endcsname}
}
\let\oldappendix\appendix 
\renewcommand\appendix{%
	\oldappendix
	\newcommand{\section@cntformat}{\appendixname~\thesection\quad}
}
\makeatother

\allowdisplaybreaks

\title{Chromatic number of randomly augmented graphs}
\author{Jan Geest and Anand Srivastav}
\begin{document}
\maketitle
\begin{abstract}
	An extension of the Erd\H{o}s-Renyi random graph model $\Gnp{n}{p}$ is the model of perturbed graphs
	introduced by Bohman, Frieze and Martin (Bohman, Frieze, Martin 2003). 
	This is a special case of the model of randomly augmented graphs studied in this paper.
	An augmented graph denoted by $\pert{H}{p}$ is the union of a deterministic host graph 
	and a random graph $\Gnp{n}{p}$.
	Among the first problems in perturbed graphs has been the question how many random edges are needed 
	to ensure Hamiltonicity of the graph. 
	This question was answered in the paper by Bohman, Frieze and Martin.
	The host graph is often chosen to be a dense graph.
	In recent years several papers on combinatorial problems in perturbed graphs were published, 
	e.g. on the emergence of powers of Hamiltonian cycles (Dudek, Reiher, Ruciński, Schacht 2020),
	some positional games played on perturbed graphs (Clemens, Hamann, Mogge, Parczyk, 2020)
	and the behavior of multiple invariants e.g. fixed clique size 
	(Bohman, Frieze, Krivelevich, Martin, 2004).
	In this paper we study the chromatic number of randomly augmented graphs.
	We concentrate on a host graph $H$ with chromatic number $o(n)$, augmented by a
	$\Gnp{n}{p}$ with $n^{-\frac{1}{3} + \delta}\leq p(n) \leq 1-\delta$ for some $\delta \in (0,1)$.
	Our main result is an upper bound for the chromatic number:
	we show that asymptotically almost surely 
	$\chromaticOf{\pert{H}{p}} \leq (1+o(1)) \cdot \frac{n \log(b)}{2 (\log(n) - \log(\chromaticOf{H}))}$
	where $b = (1-p)^{-1}$.
	This result collapses to the famous theorem of Bollobás (1988), when $H$ is the empty host graph,
	thus our result can be regarded as a generalization of the latter.
	Our proof is not constructive.
	Further, we give a constructive coloring algorithm, when the chromatic number of the host graph is at most
	$\frac{n}{\log(n)^{\alpha}},$ $\alpha>\frac{1}{2}.$

\end{abstract}

\section{Introduction}
For a graph $G=(V,E)$ the chromatic number $\chromaticOf{G}$ is the minimum number
of colors needed to color the vertices of $G$ such that no two adjacent vertices are of the same color.
Since the number of colors in this chapter is not bounded by a constant, we will use the first natural
numbers as colors, which  is a standard procedure \cite{bollobas_modern_1998}.
So for a $k\in \N$ the function $\phi : V \rightarrow [k]$ is called a $k$ (vertex) coloring of $G.$
$\phi$ is called a proper $k$-coloring of $G$ if $\{v,w\}\in E \implies \phi(v) \neq \phi(w).$
\begin{align*}{C}
	\chromaticOf{G} := \min \left\{ k : \exists \left(\phi : 
	V \rightarrow [k]\right) : \; \{v,w\}\in E  \implies \phi(v) \neq  \phi(w)\right\}.
\end{align*}
We will call a proper vertex $k$-coloring $\phi: V \rightarrow [k]$ 
just a coloring of $G,$ if clear from the context.
Let $\phi: V \rightarrow [k]$ be a proper coloring of $G.$
Then, for each $j\in[k]$ $\phi^{-1}(j)$ is called the $j$-color class of $G$ with respect to $\phi.$ 
By definition each color class is an independent set in $G.$

For any function $p: \N \rightarrow [0,1]$ and any $n\in\N$ let $\Gnp{n}{p}$ denote the random graph model
with vertex set $[n]:= \N_{\leq n},$  iwhere $ \Pof{\{v,w\} \in E(\Gnp{n}{p})} = p,$ independently
for each $ \{v,w\} \in \binom{[n]}{2}$.
We further define $b:=b(n)= \frac{1}{1-p(n)}$.

\subsection{Previous Work}
For the theory of the chromatic number in $\Gnp{n}{p}$ we refer to the pioneering papers of Erd\H{o}s and 
Rényi \cite{erdos_random_1959}, Bollobás \cite{bollobas_chromatic_1988}, Shamir and Spencer \cite{shamir_sharp_1987},
Łuczak \cite{luczak_chromatic_1991} and more recently the non-concentration result of Heckel and
Riordan\cite{heckel_how_2021}, 
as well as the books of Bollobás \cite{bollobas_random_1985} and 
Frieze and Karoński \cite{frieze_introduction_2016}.

The fact that the chromatic  number of $\Gnp{n}{p}$ for large $p$, 
i.e. $c\geq p=p(n) \geq n^{-\theta}$ for some constants
$\theta \in (0,1/3)$ and $c\in(0,1),$ is equal to 
$(1+o(1)) \frac{\log(b) n }{2\log(n)}$ $\aas$
was first proved by Bollobás in \cite{bollobas_chromatic_1988}.
Łuczak \cite{luczak_chromatic_1991} proved that the same asymptotic bounds hold
for $\chromaticOf{\Gnp{n}{p}}$, for all $p=p(n) > d/n$ with some $d\in \R$.
In \cite{shamir_sharp_1987} Shamir and Spencer showed that the chromatic number of $\Gnp{n}{p}$ for each $p$
is concentrated around its expected value within a large deviation of essentially $\sqrt{n}$.
Recently, even a long open question about the non-concentration of the chromatic number of $\Gnp{n}{1/2}$ has 
been answered by Heckel and Riordan \cite{heckel_how_2021}.

In extension of the random graph model $\Gnp{n}{p}$ the interaction between random graphs and deterministic graphs 
has been studied.
One way to analyze the interactions between a random graph and a deterministic graph is to analyze
the resilience of the random graph with respect to certain invariants, 
asking what kind of graph can be added to the $\Gnp{n}{p}$ without changing the asymptotics of 
the invariant under consideration.
Alon and Sudakov \cite{alon_increasing_2010} studied this question for the chromatic number of deterministic
graphs of bounded maximum degree and edge number.
The concept of the resilience of a graph invariant is not restricted to random graphs, 
but can be applied to deterministic graphs as well.

Another way of analyzing the interaction between a deterministic graph and a random graph
is the model of perturbed or augmented graphs.
In this model a deterministic host graph $H$ is augmented (or perturbed) by a $\Gnp{n}{p}$,
so the graph $H\cup \Gnp{n}{p}$ is studied.
Generally whether we speak of a "perturbed" or "augmented" graph depends on the probability $p$ used for 
$\Gnp{n}{p}$, and we may speak about perturbed graphs for smaller $p$ and of augmented graphs for larger $p$.
This model was first introduced by Bohman et al. in \cite{bohman_how_2003} considering the question of Hamiltonicity,
and was further studied for example by Dudek et al. in\cite{dudek_powers_2020} and 
Das et al. in \cite{das_vertex_2019}.

\subsection{Our Results}
In this paper we study the chromatic number of a randomly augmented graph $\pert{H}{p},$ where $H$ is a deterministic
host graph with chromatic number $\oOf{n},$ augmented by a $\Gnp{n}{p}$ for 
$n^{-\frac{1}{3} +\delta} \leq p \leq (1-\delta)$ for any fixed $\delta>0.$

In \autoref{sec:Upper_Bound}, we will show  that for each $\epsilon>0$
\begin{align*}
	\label{main_ineq}
	\chromaticOf{\pert{H}{p}} \leq (1+\epsilon) \cdot \frac{n \log(b)}{2 (\log(n) - \log(\chromaticOf{H}))} \quad \aas 
	\tag{N1}
\end{align*}
for a constant $p.$

This is done by estimating the expected value of the chromatic number of the
random graphs induced by the large color classes of an optimal coloring of $H$ and
adding them up.
Additionally we estimate the number of vertices in small color classes to be negligible.
Using this upper bound on the expected value of the chromatic number and combining it 
with McDiarmid's bounded differences inequality we get our main result.\\
The challenge is to show (\ref{main_ineq}) for non-constant $p.$
In fact we prove (\ref{main_ineq}) for $p=p(n)=n^{-\theta}$
where $\theta\in\left(0,\frac{1}{3}\right)$ arbitrary but fixed.
In this situation the expected chromatic number of $\Gnp{n}{p}$ is only known to be smaller than
$\frac{n\log(b)}{2\log(n)}$ for $p(n) \geq n^{-\theta}$ with 
$\theta\in\left(0,\frac{1}{3}\right).$
Let $G^\prime$ be a new subgraph  of $\Gnp{n}{p}$ of size $k<n.$
$G^\prime$ is a $\Gnp{k}{p(n)}$ random graph.
If $k$ is to small against $n,$ we cannot apply the Bollowbás result to $G^\prime,$
because $G^\prime$ is not a $\Gnp{k}{p(k)}$ random graph.
This is true in particular if $k^{-1/3}>p(n).$
So, for color classes induced by small color classes of an optimal coloring of $H$
we cannot use the result of Bollobás.
Thus we use a different approach, namely the well-known greedy color algorithm.
In the proof we will need that the color classes are of size at least
$n^{3\theta}$ on average.
To ensure this we restrict ourselves to host graphs with $\chromaticOf{H} \leq n^{1-3\theta}.$

Again the proof will be finalized by C. McDiarmid's bounded differences
inequality.

Since in both cases the bounded differences inequality was used the proofs
are non-constructive. 
We further give constructive algorithms to construct
a proper coloring with the bound in (\ref{main_ineq}), but with host graphs $H$ satisfying
$\chromaticOf{H}\leq \frac{n}{\log(n)^{1/2}}$  if $p~\in~(0,1)$ is constant.
The construction relies on a slightly improved version of the coloring of $\Gnp{n}{p}$
proposed by Bollobás \cite{bollobas_chromatic_1988}.
However in the case $p(n) = n^{-\theta}$ we already assumed that $\chromaticOf{H}\leq n^{1-3\theta},$
which suffices to construct a coloring with an upper bound as in (\ref{main_ineq}).

In \autoref{sec:conclusion_chromatic} we discuss the results and pose a few open problems 
not tackled in this paper.
For example we leave open the question about a general upper bound for $\chromaticOf{\pert{H}{p(n)}}$
where $p=p(n)\leq n^{-1/3}.$

\section{Bounding the chromatic number of $\pert{H}{p}$}\label{sec:Upper_Bound}
We use a basic result for the chromatic number of the union of any two graphs.

\begin{lemma}[Divide and Color]\label{lem:divide_and_color}
	Let $G:=(V,E_G)$ and $H:=(V,E_H)$ be graphs. 
	Further let $\Call\subseteq\mathcal{P}(V)$ be a partition of $V$ such that every $S\in\Call$ is an independent set in $G$.
	Then
	\begin{align*}
		\chromaticOf{G \cup H} \leq \sum\limits_{S\in\Call} \chromaticOf{H\left[S\right]}.
	\end{align*}
	If $G$ is a complete $k$-partite graph and $\Call$ is a $k$-partition of the vertices of $G$ as above, then
	\begin{align*}
		\chromaticOf{G \cup H} = \sum\limits_{S\in\Call} \chromaticOf{H\left[S\right]}.
	\end{align*}
\end{lemma}
The proof is kind of folklore. For readers convenience we state it briefly in \autoref{chap:chromatic_Gnp}.
We will combine the above lemma with the following fact.
\begin{remark}\label{rem:indep_is_gnp}
	Let $H$ be a deterministic graph and $p:\N \rightarrow [0,1]$.
	Let $S\subseteq V$ be a color class of a proper coloring of $H$. 
	Then $H$ does not contain edges with both endpoints in $S$.
	Thus the only edges in $\pert{H}{p}[S]$ are random edges and 
	since the edges of $\Gnp{n}{p}$ are chosen independently at random
	we can identify $\pert{H}{p}[S]$ with $\Gnp{\abs{S}}{p}.$
\end{remark}
We further remind the reader of the following well-known facts:
\begin{remark}\label{rem:log_properties}
	Let $p=p(n)\in(0,1)$,  $b= b(n) =\frac{1}{1-p(n)}$ and $\log_b(n) = \frac{\log(n)}{ \log(b)}$.
	\begin{align}
		\text{If }
		p = p(n)\rightarrow 0\text{ then }
		\log(b) \sim p. \label{eq:log_properties}
	\end{align}
	This can be shown by applying L'Hopital's rule to the term $\frac{x}{\log\left((1-x)^{-1}\right)}$ for $x\rightarrow 0.$
\end{remark}

\subsection{The Case $p\in (0,1)$ is constant}
\label{ssec:upper_bound_constant}
In this subsection we consider all $p$, where $p\in(0,1)$ is a constant.

First, we bound the number of vertices of the host graph in small color classes.
Here the adjective small shall refer to color classes that are significantly smaller than the average color class,
while the adjective large will be used to describe color classes which are not small.
\begin{lemma}\label{lem:chromatic_constant_p_small_component}
	Let $H=([n],E)$ be a graph with $\chromaticOf{H} = \frac{n}{\beta(n)}$ for some function  $\beta:\N\rightarrow \R$ with  $\beta(n)\rightarrow\infty$ as $n\rightarrow \infty.$
	Furthermore let $\Call$ be the set of all color classes of an optimal coloring of $H$.
	For $\epsilon>0$ set $g(n) = \beta(n)^{\left(1+\frac{\epsilon}{2}\right)^{-\frac{1}{2}}}$ and define $\Clarge{g}=~\left\{S \in \Call \vert \abs{S} \geq g(n) \right\}$ and $\Csmall{g} = \Call \setminus \Clarge{g}$.
	For sufficiently large $n=n(\epsilon)$  we have
	\begin{align*}
		\abs{\bigcup_{S\in \Csmall{g}} S} \leq \frac{\epsilon}{2}\cdot \frac{\log(b) n}{2 \log(\beta(n))}.
	\end{align*}
	for every $b$ where $b = \frac{1}{1-p}$ for some constant $p\in (0,1).$
\end{lemma}
\begin{proof}
	Since the sets $S\in\Csmall{g}$ are pairwise disjoint and $(1+\epsilon)^{-\frac{1}{2}} < 1$,
	we get the following estimates
	\begin{align*}
		\abs{\bigcup_{S\in \Csmall{g}} S} 
		=	& \sum_{S\in \Csmall{g}} \abs{S}\\
		\leq& \chromaticOf{H}g(n)\\
		=	& \frac{n}{\beta(n)} \cdot \beta(n)^{\left(1+\frac{\epsilon}{2}\right)^{-\frac{1}{2}}}\\
		=	& n \cdot \beta(n)^{\left(1+\frac{\epsilon}{2}\right)^{-\frac{1}{2}}-1} 
		& \tag{Note $\left(1+\frac{\epsilon}{2}\right)^{-\frac{1}{2}}-1 < 0$}\\
		\leq& \frac{\epsilon}{2}\cdot \frac{\log(b) n}{2 \log(\beta(n))}. 
	\end{align*}
	The last inequality can be showed as follows.
	Let $x:=\beta(n)$, $c:= \frac{\epsilon \log(b)}{4}$ and 
	$a := \left(1+\frac{\epsilon}{2}\right)^{-\frac{1}{2}}-1$.
	Then for sufficiently large $x,$ we have $x^a \leq \frac{c}{\log(x)}$.\\ 
\end{proof}

We need  the following implication of a result of Bollobás \cite{bollobas_chromatic_1988}.
Its proof is given in the appendix.

\begin{corollary}\label{cor:chromatic_gnp_large_p_expectation}
	Let $p\in(0,1)$ be constant.
	For any $\epsilon >0$ and sufficiently large $n=n(\epsilon)\in \N$ we have
	\begin{align*}
		\Eof{\chromaticOf{\Gnp{n}{p}}} \leq (1+\epsilon) \frac{n\log(b)}{2\log(n)}.
	\end{align*}
\end{corollary}

We are ready to give an upper bound for the expected value of the chromatic number of an augmented graph.

\begin{theorem}\label{thm:chromatic_constant_p_expectation}
	Let $H=([n],E)$ be a deterministic graph with $\chromaticOf{H} = \frac{n}{\beta(n)}$ for some function 
	$\beta: \N \rightarrow \N$ such that $\beta(n) \rightarrow \infty$ for $n \rightarrow \infty$.
	Let further $p\in(0,1)$ be constant and $b := \frac{1}{1-p}$.
	Then for each $\epsilon>0$ there is $n(\epsilon) \in \N$ such that for all $n\geq n(\epsilon)$
	\begin{align*}
		\Eof{\chromaticOf{\pert{H}{p}}} \leq (1+\epsilon) \cdot \frac{n \log(b)}{2 (\log(n) - \log(\chromaticOf{H}))}.
	\end{align*}
\end{theorem}

\begin{proof}
	Let $\Call$ denote the set of all color classes of an optimal coloring of $H$ and let $\epsilon>0$.
	Additionally we define for $g(n) = \beta(n)^{\left(1+\frac{\epsilon}{2}\right)^{-\frac{1}{2}}}$ the set of large
	color classes $\Clarge{g}=~\left\{S \in \Call \vert \abs{S} \geq g(n) \right\}$ and the set of small color 
	classes $\Csmall{g} = \Call \setminus \Clarge{g}$ as in \autoref{lem:chromatic_constant_p_small_component}.
	Note that a set $S\in \Csmall{g} \cup \Clarge{g}$ is an independent set in $H$, because 
	$S$ is a color class of an optimal (proper) coloring of $H$.
	So, the only edges in $\pert{H}{p}[S]$ are the edges of $\Gnp{n}{p}$.
	We can thus identify $\pert{H}{p}[S]$ with $\Gnp{\abs{S}}{p}$ according to \autoref{rem:indep_is_gnp}.
	Furthermore, since $p$ is constant the assumptions of \autoref{cor:chromatic_gnp_large_p_expectation} are
	fulfilled for sufficiently large $n$.
	Now
	\begin{align*}
		\Eof{\chromaticOf{\pert{H}{p}}} 
		\leq& \sum_{S\in \Call} \Eof{\chromaticOf{\pert{H}{p}[S]} }\\
		=& \sum_{S\in\Clarge{g}} \Eof{\chromaticOf{\pert{H}{p}[S]}}  
		+\sum_{S\in\Csmall{g}} \Eof{\chromaticOf{\pert{H}{p}[S]}}\\
		=& \sum_{S\in\Clarge{g}} \Eof{\chromaticOf{\Gnp{\abs{S}}{p}}}  
		+\sum_{S\in\Csmall{g}} \underbrace{\Eof{\chromaticOf{\Gnp{\abs{S}}{p}}}}_{\leq \abs{S}} \\
		\leq& \sum_{S\in\Clarge{g}} \Eof{\chromaticOf{\Gnp{\abs{S}}{p}}}  
		+\sum_{S\in\Csmall{g}} \abs{S} \\
		\leq& \sum_{S\in\Clarge{g}} \left(1+ \frac{\epsilon}{2}\right)^{\frac{1}{2}} \cdot \frac{\log(b) 
			\abs{S}}{2 \log(\abs{S})}+\sum_{S\in\Csmall{g}} \abs{S} \\
		\tag{for sufficiently large $n$ by \autoref{cor:chromatic_gnp_large_p_expectation}}\\
		\leq& \sum_{S\in\Clarge{g}} \left(1+ \frac{\epsilon}{2}\right)^{\frac{1}{2}} \cdot 
		\frac{\log(b) \abs{S}}{2 \log(\abs{S})} 
		+ \frac{\epsilon}{2} \cdot \frac{n \log(b)}{2\log(\beta(n))}
		\tag{\autoref{lem:chromatic_constant_p_small_component}}\\
		\leq& \left(1+\frac{\epsilon}{2}\right)^{\frac{1}{2}}\cdot \frac{\log(b)}{2} \cdot \sum_{S\in\Clarge{g}} \frac{\abs{S}}{ \log(g(n))} 
		+ \frac{\epsilon}{2} \cdot \frac{n \log(b)}{2\log(\beta(n))}\\
		\leq& \left(1+\frac{\epsilon}{2}\right) \cdot \frac{\log(b)}{2 \log(\beta(n))}
		\underbrace{\sum_{S\in\Clarge{g}} \abs{S}}_{\leq n}
		+ \frac{\epsilon}{2} \cdot \frac{n \log(b)}{2\log(\beta(n))}\\
		\leq& (1+\epsilon) \cdot \frac{n \log(b)}{2\log(\beta(n))}\\
		= 	& (1+\epsilon) \cdot \frac{n \log(b)}{2(\log(n)-\log(\chromaticOf{H}))}.
	\end{align*}
\end{proof}

We invoke the bounded differences inequality of C. McDiarmid \cite{siemons_method_1989}:
\begin{theorem}[McDiarmid's inequality]
	\label{thm:mcdiarmid}
	Let $X_1,\ldots,X_n$ be independent random variables with $X_j$ taking values in some set $A_j$.
	Suppose that some measurable function $f:\prod_{j=1}^{n} A_j \rightarrow \R$
	satisfies
	\begin{align*}
		\abs{f(x) - f(x^{\prime})} \leq c_j
	\end{align*}
	whenever the vectors $x,x^{\prime} \in \prod_{j=1}^{n} A_j$ differ only in the $j$-th component.
	Let $Y$ be the (real valued) random variable $Y:=f(X_1,\ldots,X_n)$.
	Then for any $t>0$
	\begin{align*}
		\Pof{\abs{Y -\Eof{Y}} \geq t} \leq 2 \exp\left( \frac{-2t^2}{\sum_{j=1}^n c_j^2} \right).
	\end{align*}
\end{theorem}
\begin{remark}
	In order to apply McDiarmid's inequality to the chromatic number of $\pert{H}{p}$, we set
	$X_j :=\{\{i,j\} \in E(\pert{H}{p}) \vert i<j \}$.
	Since the $X_j$ form a partition of the edges of $\pert{H}{p}$, 
	the random variable $\chromaticOf{\pert{H}{p}}$ depends solely on the values of the $X_j$.
	We define
	\begin{align*}
		Y = f(X_1,\ldots, X_n) = \chromaticOf{\pert{H}{p}}.
	\end{align*}
	Note that the $X_j$ are mutually independent random variables.
	Let $X=(X_1,\ldots , X_n)$
	Furthermore for each $j$, if $X$ and $\hat{X}$ only differ in the $j$-th component, we have
	\begin{align*}
		\abs{f(X) - f(\hat{X})} \leq 1,
	\end{align*}
	since one additional color will always be enough to color the vertex $j$, if necessary.
	Thus McDiarmid's inequality is applicable with $\sum_{j=1}^n c_j^2 = n$ and we get the following
	upper bound for the chromatic number of $\pert{H}{p}.$
\end{remark}

\begin{theorem}\label{thm:chromatic_constant_p_aas}
	Let $H=([n],E)$ be a deterministic graph with $\chromaticOf{H} = \frac{n}{\beta(n)}$ for some function $\beta: \N \rightarrow \N$ such that $\beta(n) \rightarrow \infty$ for $n \rightarrow \infty$.
	Let further be $p\in(0,1)$ and $b = \frac{1}{1-p}$.
	Then for each $\epsilon >0$ there exists $n(\epsilon)\in \N$ such that for all 
	$n\geq n(\epsilon)$ we have
	\begin{align*}
		\chromaticOf{\pert{H}{p}} \leq (1+\epsilon) \cdot \frac{n \log(b) }{2(\log(n) -\log(\chromaticOf{H}))} \quad \aas
	\end{align*}	
\end{theorem}

\begin{proof}
	Let $\epsilon >0$ and $\lambda= \frac{n \log(b)}{2\left(\log(n)-\log(\chromaticOf{H})\right)}$.
	By \autoref{thm:chromatic_constant_p_expectation} and McDiarmid's inequality we get
	\begin{align*}
		&\Pof{\chromaticOf{\pert{H}{p}} \geq (1+\epsilon) \cdot \lambda}\\
		=	&\Pof{\chromaticOf{\pert{H}{p}} - \left(1+\frac{\epsilon}{2}\right) \cdot \lambda \geq \frac{\epsilon}{2} \cdot\lambda }\\
		\leq& \Pof{\chromaticOf{\pert{H}{p}} - \Eof{\chromaticOf{\pert{H}{p}}} \geq \frac{\epsilon}{2} \cdot \lambda}
		\tag{\autoref{thm:chromatic_constant_p_expectation}}\\
		\leq& 2\exp\left(-\frac{2 }{n} \cdot\left(\frac{\epsilon \lambda}{2}\right)^2\right)  \tag{McDiarmid's inequality}\\
		\leq&2\exp\left(-\frac{2}{n}\cdot \left(\frac{\epsilon n \log(b)}{4(\log(n) - \log(\chromaticOf{H}))}\right)^2 \right)  \\
		=   & 2\exp\left(- \frac{\epsilon^2 \log(b)^2 n }{8 (\log(n) - \log(\chromaticOf{H}))^2}\right) = o(1).
	\end{align*}
\end{proof}

\subsection{The Case $p(n) = n^{-\theta}$ for some $\theta \in \left(0,\frac{1}{3}\right)$}

We now consider $\pert{H}{p}$ with non-constant and small $p$.
When trying to apply the same strategy as for constant $p$, it turns out that for small color classes it is not enough to merely count the 
number of vertices as in \autoref{lem:chromatic_constant_p_small_component}.
The reason is that the number of vertices in small components can be of a higher order
than the desired bound for the chromatic number of $\pert{H}{p}$. 
Instead we will use another implication of a result of Bollobás \cite{bollobas_chromatic_1988}.
Again its explicit proof can be found in the appendix.
\begin{corollary}\label{cor:chromatic_gnp_small_p_expectation}
	Let $\theta\in\left(0,\frac{1}{3}\right)$, $\delta>0 $ and $n^{-\frac{1}{3} + \delta} \leq p(n) \leq n^{-\theta}$. 
	Let $\epsilon>0$.
	For sufficiently large $n\in \N$ we have
	\begin{align*}
		\Eof{\chromaticOf{\Gnp{n}{p}}} \leq (1+\epsilon) \frac{np}{2\log(np)}.
	\end{align*}
\end{corollary}
This implies the following result, which bounds the number of colors needed to color the small color classes of $H$.
\begin{lemma}\label{lem:chromatic_small_p_small_components}
	Let $\epsilon >0$ and $H = ([n],E)$ be a deterministic graph with $\chromaticOf{H} = \frac{n}{\beta(n)}$ for some function $\beta : \N \rightarrow \Q$.
	Let further $g(n) = \frac{\beta(n)}{\log(n)}$ and let $\Call$ the set of all color classes of an optimal coloring of $H$.
	Define  $\Clarge{g}=~\left\{S \in \Call \vert \abs{S} \geq g(n) \right\}$ and $\Csmall{g} = \Call \setminus \Clarge{g}$.
	Let further $n=n(\epsilon)$ sufficiently large and $p(n) = n^{-\theta}$, where $\theta\in \left(0 , \frac{1}{3}\right)$ and $b=\frac{1}{1-p}$.
	If $\beta(n) \geq n^{3\theta}$, then we have for $n\geq n(\epsilon)$ for some $n(\epsilon)\in \N,$
	\begin{align*}
		\Eof{\chromaticOf{\pert{H}{p}\left[{\bigcup_{ C \in \Csmall{g}}C}\right]} }  \leq \frac{\epsilon n p}{2(\log(np) -\log(\chromaticOf{H})) }.
	\end{align*}
\end{lemma}

\begin{proof}
	According to \autoref{rem:indep_is_gnp}, for each $C \in \Csmall{g}$ we 
	can identify $\pert{H}{p(n)}\left[C\right]$ with $\Gnp{\abs{C}}{p(n)}$.
	So by invoking \autoref{cor:chromatic_gnp_small_p_expectation} with $\epsilon= \frac{1}{3}$ we get
	\begin{align*}
		\Eof{\chromaticOf{\pert{H}{p(n)}[C]}} 
		=& \Eof{\chromaticOf{\Gnp{\abs{C}}{p(n)}}}\\
		\leq& \Eof{\chromaticOf{\Gnp{\abs{g(n)}}{p(n)}}} \tag{since $\abs{C} \leq g(n)$} \\
		\leq& \left(1+\frac{1}{3}\right) \frac{g(n)p(n)}{2\log(g(n)p(n))}\tag{\autoref{cor:chromatic_gnp_small_p_expectation}}\\
		=& \frac{2}{3} \frac{g(n) p(n) }{\log(g(n)p(n))} \\
		=& \frac{2}{3} \frac{\beta(n) p(n)}{\log(n) (\log(\beta(n)p(n))-\log(\log(n)))}\\
		\leq& \frac{2}{3} \frac{\beta(n) p(n)}{\log(n) (1-o(1))\log(\beta(n)p(n))}\tag{since $\beta(n)p(n) \geq n^{2\theta}$}\\
		\leq& \frac{\beta(n) p(n)}{\log(n) \log(\beta(n)p(n))}.
	\end{align*}
	Thus, since $\abs{\Csmall{g}} \leq \chromaticOf{H} \leq \frac{n}{\beta(n)}$ we have
	\begin{align*}
		\Eof{\chromaticOf{\pert{H}{p}\left[{\cup_{ C \in \Csmall{g}}C}\right]} } 
		\leq& \sum_{C\in \Csmall{g}} \Eof{\chromaticOf{\pert{H}{p}[C]}}  \tag{\autoref{lem:divide_and_color}} \\
		\leq& \frac{n}{\beta(n)} \frac{\beta(n) p(n)}{\log(n) \log(\beta(n)p)} \tag{$\abs{\Csmall{g}} \leq \frac{n}{\beta(n)}$}	\\
		=	& \frac{n p(n)}{\log(n) \log(\beta (n)p)}\\
		\leq& \frac{\epsilon n p}{2\log(\beta(n)p) } \tag{for $n$ large enough}\\
		=	&\frac{\epsilon n p}{2(\log(np) -\log(\chromaticOf{H})) }.
	\end{align*}
\end{proof}

Next, we bound the expectation of $\chromaticOf{\pert{H}{p}}$ from above using 
\autoref{lem:chromatic_small_p_small_components}.
\begin{theorem}\label{thm:chromatic_small_p_expectation}
	Let $H = ([n], E)$ be a deterministic graph with $\chromaticOf{H} = \frac{n}{\beta(n)}$ for some function $\beta : \N \rightarrow \Q $.
	Let further $\epsilon>0$ and $ p(n) = n^{-\theta}$, where $\theta \in \left(0 , \frac{1}{3}\right)$ and $b=\frac{1}{1-p}$.
	If $\beta(n) \geq n^{3(\theta+\delta)}$ for some constant $\delta >0$, then there exists $n(\epsilon) \in \N$
	so that for all $n\geq n(\epsilon)$ we have
	\begin{align*}
		\Eof{\chromaticOf{\pert{H}{p}}} \leq (1+\epsilon) \cdot \frac{n p}{2 (\log(np) - \log(\chromaticOf{H}))}.
	\end{align*}
\end{theorem}
\begin{proof}
	Let $\Call$ denote the set of all color classes of an optimal coloring of $H$ and let $ \epsilon>0$.
	Additionally we define for $g(n) = \frac{\beta(n)}{\log(n)}$ the set of large color classes 
	$\Clarge{g} = \{S\in\Call \vert \abs{S} \geq g(n)\}$ and the set of small color classes $\Csmall{g} = \Call \setminus \Clarge{g}$.\\
	For a large color class $S\in \Clarge{g}$ of an optimal coloring of $H$ there are no edges of $H$ within $S,$
	so we can identify $\pert{H}{p(n)} [S]$ with $\Gnp{\abs{S}}{p(n)},$ by \autoref{rem:indep_is_gnp}.
	Since $\abs{S}\geq g(n) = \frac{\beta(n)}{\log(n)} \geq n^{3\left(\theta+\frac{\delta}{2}\right)}$
	for sufficiently large $n,$ we have
	$$ 
	p(n) = n^{-\theta} 
	\geq \abs{S}^{-\frac{\theta}{3(\theta +\delta/2)}} = \abs{S}^{-\frac{1}{3}+\frac{\delta}{6(\theta+\delta/2)}},
	$$ 
	and we may invoke \autoref{cor:chromatic_gnp_small_p_expectation}
	with $\epsilon^\prime = \left(1+ \frac{\epsilon}{2}\right)^{\frac{1}{2}}-1 >0$
	and get
	\begin{align}
		\Eof{\chromaticOf{\pert{H}{p}[S]}} 
		=\Eof{\chromaticOf{\Gnp{\abs{S}}{p(n)} } }
		\leq \left(1+\frac{\epsilon}{2}\right)^{\frac{1}{2}}  \frac{p \abs{S}}{2 \log(\abs{S}p)}.
		\label{ineq:small_p_expectation_large_classes}
	\end{align}
	For sufficiently large $n,$ $ \log(\log(n)) \in \oOf{(\log(\beta(n)))}.$ 
	Thus
	\begin{align}
		\left(1+\frac{\epsilon}{2}\right)^{\frac{1}{2}}\log(g(n))
		= \left(1+\frac{\epsilon}{2}\right)^{\frac{1}{2}} (\log(\beta(n)) -
		\underbrace{\log(\log(n))}_{=\oOf{\log(\beta(n))}})
		\geq \log(\beta(n)). \label{ineq:small_p_beta_g}
	\end{align}
	Now
	\begin{align*}
		\Eof{\chromaticOf{\pert{H}{p}}}
		\leq&\sum_{S\in \Call} \Eof{\chromaticOf{\pert{H}{p}} [S]} \tag{\autoref{lem:divide_and_color}}\\
		=	&\sum_{S\in\Clarge{g}} \Eof{\chromaticOf{\pert{H}{p}[S]}} + 
		\sum_{C\in \Csmall{g}} \Eof{\chromaticOf{\pert{H}{p}[S]}} \\
		= 	&\sum_{S\in\Clarge{g}} \Eof{\chromaticOf{\Gnp{\abs{S}}{p}}} + 
		\sum_{C\in \Csmall{g}} \Eof{\chromaticOf{\Gnp{\abs{S}}{p}}}\tag{\autoref{rem:indep_is_gnp}} \\
		\leq& \sum_{S\in\Clarge{g}}
		\left(1+\frac{\epsilon}{2}\right)^{\frac{1}{2}}  \frac{p \abs{S}}{2 \log(\abs{S}p)} + 
		\sum_{C\in \Csmall{g}} \Eof{\chromaticOf{\Gnp{\abs{S}}{p}}}
		\tag{with (\ref{ineq:small_p_expectation_large_classes})}\\
		\leq & \sum_{S\in\Clarge{g}} \left(1+\frac{\epsilon}{2}\right)^{\frac{1}{2}}  \frac{p \abs{S}}{2 \log(\abs{S}p)} + 
		\frac{\epsilon}{2}\cdot \frac{n p}{2 \log(\beta(n)p)}\\
		\tag{\autoref{lem:chromatic_small_p_small_components} for $n\geq n(\epsilon)$ for some $n(\epsilon)\in\N$}\\
		\leq& \sum_{S\in\Clarge{g}} \left(1+\frac{\epsilon}{2}\right)^{\frac{1}{2}}  \frac{p \abs{S}}{2 \log(g(n)p)} + \frac{\epsilon}{2}\cdot \frac{n p}{2 \log(\beta(n)p)}\tag{$|S|\geq g(n)$}\\
		=	& \left(1+\frac{\epsilon}{2}\right)^{\frac{1}{2}}  \frac{p \sum_{S\in\Clarge{g}}\abs{S}}{2 \log(g(n)p)} + \frac{\epsilon}{2}\cdot \frac{n p}{2 \log(\beta(n)p)}\\
		\leq& \left(1+\frac{\epsilon}{2}\right)^{\frac{1}{2}} \frac{p n}{2 \log(g(n)p)} + \frac{\epsilon}{2}\cdot \frac{n p}{2 \log(\beta(n)p)}\\
		\leq& \left(1+\frac{\epsilon}{2}\right) \frac{p n}{2 \log(\beta(n)p)} + \frac{\epsilon}{2}\cdot \frac{n p}{2 \log(\beta(n)p)}
		\tag{with (\ref{ineq:small_p_beta_g})}\\
		=& \left(1+\epsilon\right) \frac{p n}{2 (\log(np) -\log(\chromaticOf{H}))}
	\end{align*}		
\end{proof}

The main result of this section is the following theorem, where we use McDiarmid's bounded differences inequality
(\autoref{thm:mcdiarmid})
to bound the probability that $\chromaticOf{\pert{H}{p}}$ differs from its expectation by a large amount.
\begin{theorem}\label{thm:chromatic_small_p_aas}
	Let $H = ([n], E)$ be a deterministic graph with $\chromaticOf{H} = \frac{n}{\beta(n)}$ for some function $\beta : \N \rightarrow \Q .$
	Let further $p(n) \geq n^{-\theta}$, where $\theta \in \left(0 , \frac{1}{3}\right)$ and $b=\frac{1}{1-p}$.
	If $\beta(n) \geq n^{3\theta}$, then
	\begin{align*}
		\chromaticOf{\pert{H}{p}} \leq (1+\epsilon) \cdot \frac{n p}{2 (\log(np) - \log(\chromaticOf{H}))} \qquad\aas
	\end{align*}
\end{theorem}
\begin{proof}
	Let $\epsilon >0$ and $\lambda=\frac{\epsilon n \log(b)}{4(\log(np) - \log(\chromaticOf{H}))}.$
	We have
	\begin{align*}
		&\Pof{\chromaticOf{\pert{H}{p}} \geq (1+\epsilon) \frac{n \log(b)}{2(\log(np) - \log(\chromaticOf{H}))}}\\
		=	&\Pof{\chromaticOf{\pert{H}{p}} \geq \left(1+\frac{\epsilon}{2}\right) \frac{n \log(b)}{2(\log(np) - \log(\chromaticOf{H}))} +\lambda}\\
		\leq&\Pof{\chromaticOf{\pert{H}{p}} \geq \Eof{\chromaticOf{\pert{H}{p}}} + \lambda}\tag{\autoref{thm:chromatic_small_p_expectation}}\\
		\leq&2\exp\left(- \frac{\lambda^2}{2n}\right)\tag{McDiarmid's inequality, \autoref{thm:mcdiarmid}}\\
		\leq&2\exp\left(-\left(\frac{\epsilon^2 n p^2}{32 (\log(n) -\log(\chromaticOf{H}))^2}\right)\right)
		\tag{by \autoref{rem:log_properties} and $\log(np) \leq \log(n)$}\\
		\leq&\exp\left(-n^{\frac{1}{4}}\right),
	\end{align*}
	and the last inequality holds, because $p\geq n^{-\theta}$, $\theta \in (0,\frac{1}{3})$ implies $np^2 \geq n^{1\!/\!3}$.
\end{proof}

\subsection{Tightness of upper bounds for certain graphs}
We proved an upper bound for the chromatic number of an augmented graph,
depending only on $p$ and the chromatic number of the host graph.
A natural question is, of course, whether our bounds are tight.
Indeed, the next theorem shows, that there exist host graphs for which the asymptotics for 
the chromatic number is equal to the asymptotics of the upper bound in \autoref{thm:chromatic_constant_p_aas}.
We will prove this by considering the independence number of a certain class of host graphs.
These host graphs are characterized by a "low" number of "large" independent sets.
As an example one can consider the complete $\chromaticOf{H}$-partite graphs.
\begin{lemma}\label{lem:lower_bound_indep_num}
	Let $H = ([n],E)$ be a deterministic graph with $\chromaticOf{H} = \frac{n}{\beta(n)}$ for some function $\beta=\beta(n)$ with $\beta(n) \rightarrow \infty$ as $n\rightarrow\infty$.
	Further let $k = \ceilOf{\frac{2(\log(n) -\log(\chromaticOf{H}))}{\log(b)}}$ and 
	$n_{H,k}$ the number of independent sets of size $k$ in $H$.\\
	If $n_{H,k} \in o\left(\left(\frac{\chromaticOf{H}}{n}\right)^{-k+1}\right)$,
	then $\alpha(\pert{H}{p})\leq \frac{2(\log(n) -\log(\chromaticOf{H}))}{\log(b)}$ \aas
\end{lemma}
\begin{proof}
	Let $X_k$ be the random variable that counts the number of independent sets of size $k$ in $\pert{H}{p}$.
	Now
	\begin{align*}
		\Pof{\alpha(\pert{H}{p}) \geq k} 
		&= \Pof{X_k \geq 1}\\
		&\leq \Eof{X_k} \tag{Markov's inequality}\\
		&= n_{H,k} (1-p)^{\binom{k}{2}}\\
		&=n_{H,k}\cdot\exp\left(- \frac{\log(b)}{2} k (k-1)\right)\tag{$b=(1-p)^{-1}$}\\
		&\leq n_{H,k}\cdot \exp\left(- \frac{\log(b)}{2} \frac{2(\log(n) -\log(\chromaticOf{H}))}{\log(b)} (k-1)\right)\\
		&= n_{H,k} \cdot \left(\frac{\chromaticOf{H}}{n}\right)^{-k+1}\\
		&= o(1),
	\end{align*}
	with our assumption on $n_{H,k}$.
\end{proof}
Together with the well known inequality $\chromaticOf{G} \geq \frac{n}{\alpha(G)}$ 
\autoref{lem:lower_bound_indep_num} implies \autoref{thm:lower_bound_aas}:
\begin{theorem}\label{thm:lower_bound_aas}
	Let $H = ([n],E)$ be a deterministic graph with $\chromaticOf{H} = \frac{n}{\beta(n)}$ 
	for some function $\beta =\beta(n)$ with $\beta(n) \rightarrow \infty$ as $n\rightarrow \infty$.
	Further let $k = \ceilOf{\frac{2(\log(n) -\log(\chromaticOf{H}))}{\log(b)}}$ and 
	$n_{H,k}$ the number of independent $k$ sets, in $H$.\\
	If $n_{H,k} \in o\left(\left(\frac{\chromaticOf{H}}{n}\right)^{-k+1}\right)$,
	then for each $\epsilon>0$ there is a $n(\epsilon)\in \N$, so that for all $n \geq n(\epsilon)$ we have
	\begin{enumerate}
		\item  $\alpha(\pert{H}{p})\leq \frac{2(\log(n) -\log(\chromaticOf{H}))}{\log(b)}$ \aas 
		
		\item 
		\(
		\chromaticOf{\pert{H}{p}} \geq (1-\epsilon) \cdot \frac{n \log(b) }{2(\log(n) -\log(\chromaticOf{H}))}
		\) \aas
	\end{enumerate}
\end{theorem}

\section{Coloring augmented graphs with host graphs of small chromatic number}\label{sec:construction}
The proofs of \autoref{thm:chromatic_constant_p_aas} and \autoref{thm:chromatic_small_p_aas} 
rely on  McDiarmid's bounded differences inequality and are not constructive.
In this section we will give algorithms based on algorithms of Bollobás \cite{bollobas_chromatic_1988}
(in the following \autoref{alg:chrom_constant_p_Gnp} and \autoref{alg:chrom_small_p_Gnp}). 
They find a coloring of $\pert{H}{p}$ 
with at most
$(1~+~\epsilon)~\cdot~\frac{n\log(b)}{2(\log(n) -\log(\chromaticOf{H}))}$ colors \aas{}
We restrict ourselves to host graphs $H$ with $\chromaticOf{H} \leq \frac{n}{\log(n)^\gamma}$ 
for some constant $\gamma \geq \frac{1}{2}$ in the case that $p$ is constant and
host graphs with  $\chromaticOf{H} \leq n^{\frac{1-3\theta-2\delta}{2}}$ for some
arbitrary small constant $\delta>0$ if $p(n) = n^{-\theta}$.
In preparation for the rest of the section we define the sets of large and small color classes 
analogously to \autoref{sec:Upper_Bound}.
The Algorithms \ref{alg:chrom_constant_p_pertHp}, \ref{alg:chrom_small_p_pertHp}, 
\ref{alg:chrom_constant_p_Gnp} and \ref{alg:chrom_small_p_Gnp} have exponential
time complexity.
\begin{definition}
	Let $H=([n],E)$ be a deterministic graph. Let $f$ be a coloring of $H$ and $\Call$ the set of all color classes with respect to $f$.
	Let further $g:\N \rightarrow \R$ be a function. 
	We define
	\begin{align*}
		\Csmall{g} := \left\{S\in \Call \vert \abs{S} < g(n)\right\}\text{ and }
		\Clarge{g} := \left\{S\in \Call \vert \abs{S} \geq g(n)\right\}.
	\end{align*}
\end{definition}
$\Csmall{g}$ is the set of small and $\Clarge{g}$ is the set of large color classes with respect to $f$.

In addition we will need the following variation of the union bound which is a basic result of stochastic.
An explicit proof can be found in the appendix.

\begin{lemma}\label{lem:union_bound}
	Let $(\Omega,\Sigma, \mathbb{P}) $ be a probability space and $k\in \N.$
	For each $i\in [k]$ let $X_i$ be a random variable and $A_i$ in $\R.$
	Then 
	\begin{align*}
		\Pof{\sum_{i = 1}^{k} X_i \leq \sum_{i = 1}^k A_i} 
		\leq \sum_{i = 1}^{k}\Pof{X_i \leq A_i}.
	\end{align*}
\end{lemma}

Bollobás \cite{bollobas_chromatic_1988} used the following algorithm to calculate an upper bound for the chromatic number of $\Gnp{n}{p}$,
where $0<p<1$ is constant.

\begin{algorithm}[H]
	\caption{Coloring for constant $0<p<1$}\label{alg:chrom_constant_p_Gnp}
	\SetAlgoLined
	\KwIn{A random graph $\Gnp{n}{p}$}
	\KwOut{A coloring of the vertices of $\Gnp{n}{p}$}
	\While{At least $\frac{n}{{\log_b(np)}^2}$ vertices remain uncolored and if there is an uncolored independent set $I$ of size $2 \log_b(np) - 4 \log_b(\log_b(np))$}{Color $I$ with a new color\;}
	Color all uncolored vertices with their own color\;
\end{algorithm}
We will use it to color the large color classes of $H$.

With a small change to the arguments of Bollobás \cite{bollobas_chromatic_1988} one can show the following theorem.
\begin{theorem}\label{thm:chromatic_constant_p_Gnp}
	Let $\delta>0$ and $(1-\delta) \geq p\gg n^{-\theta}$ for each constant $\theta \in (0,\frac{1}{3})$.
	For each $\epsilon>0$ there exists $n=n(\epsilon)\in\N$ so that for all $n>n(\epsilon)$ and a suitable constant $c>0$ such that
	the probability that \autoref{alg:chrom_constant_p_Gnp} fails to find a proper coloring of $\Gnp{n}{p}$
	with at most $(1 + \epsilon) \frac{n \log(b)}{2 \log(np)}$ colors is at most 
	$\exp\left(- \frac{c n^2 p^7}{\log(n)^8}\right).$
	Thus
	\begin{align*}
		\Pof{\chromaticOf{\Gnp{n}{p}} \geq (1 + \epsilon) \frac{n \log(b)}{2 \log(np)}} 
		\leq \exp\left(- \frac{c n^2 p^7}{\log(n)^8}\right).
	\end{align*}
\end{theorem}
This improves on the result of Bollobás \cite{bollobas_chromatic_1988}, where it was
only shown that for arbitrarily small $\epsilon>0$ and some constant $c>0$
\begin{align}
	\Pof{\chromaticOf{\Gnp{n}{p}} \geq (1 + \epsilon) \frac{n \log(b)}{2 \log(np)}} 
	\leq \exp\left(-c n^{2-\epsilon} p^8\right).
\end{align}
The explicit proof of this result can be found in the appendix.

We use the following algorithm to construct a proper coloring of $\pert Hp$.

\begin{algorithm}[H]
	\caption{Coloring $\pert{H}{p}$}\label{alg:chrom_constant_p_pertHp}
	\KwIn{A deterministic graph $H$ with $\chromaticOf{H}\leq \frac{n}{\beta(n)}$ 
		where $\beta(n)\geq \log(n)^\gamma$ for some constant $\gamma > \frac{1}{2}$, 
		an augmented graph $G=\pert{H}{p}$ for some constant $p\in(0,1)$ and 
		$\epsilon>0$ with $\epsilon<2(2\gamma -1)$.}
	\KwOut{A proper coloring of $G$ using less than  
		$(1+\epsilon) \cdot \frac{n \log(b)}{2 (\log(n) - \log(\chromaticOf{H}))}$ colors.}
	
	Construct an optimal coloring $ \chi^\prime$ of $H$\;
	In each color class $S$ of $\chi^\prime$ with 
	$\abs{S} \geq \beta(n)^{\left(1+\frac{\epsilon}{2}\right)^{-\frac{1}{2}}}=: g(n)$ 
	construct a proper coloring using \autoref{alg:chrom_constant_p_Gnp}\;
	Color each vertex that is still uncolored with its own color\;
\end{algorithm}

\begin{theorem}\label{thm:construction_p_constant}
	Let $H =([n],E)$ be a deterministic graph with $\chromaticOf{H}\leq \frac{n}{\log(n)^\gamma}$ 
	for some constant $\gamma > \frac{1}{2}$
	and let $\beta(n)=\frac{n}{\chromaticOf{H}}$.
	Let further $p\in (0,1)$ be constant and $b=\frac{1}{1-p}$.
	Then \autoref{alg:chrom_constant_p_pertHp} \aas constructs a proper coloring for $\pert{H}{p}$ 
	using at most $(1+\epsilon) \cdot \frac{n\log(b)}{2(\log(n) -\log(\chromaticOf{H}))}$ colors.
\end{theorem}
\begin{proof}
	Since this algorithm constructs a proper coloring, the only thing left to prove is that 
	it uses at most  $(1+\epsilon) \cdot \frac{n \log(b)}{2 (\log(n) - \log(\chromaticOf{H}))}$ 
	colors \aas.
	We will show that for coloring the large color classes of $H$ at most 
	$(1+\frac{\epsilon}{2}) \cdot \frac{n \log(b)}{2 (\log(n) - \log(\chromaticOf{H}))}$ 
	colors are used \aas.
	Since by \autoref{lem:chromatic_constant_p_small_component} there are at most 
	$\frac{\epsilon}{2}\cdot \frac{\log(b) n}{2 \log(n)}$ vertices of $\pert{H}{p}$ left uncolored,
	the total number of colors used is at most as stated in the theorem.

	For each $S\subseteq [n]$, let $\xi(\pert{H}{p}[S])$ be the number of colors that were used 
	by \autoref{alg:chrom_constant_p_pertHp} to color the vertices of $S$.
	First we recall some preliminary facts:
	each $S\in \Clarge{g}$ is an independent set in $H$ thus $\pert{H}{p}[S]$ can be identified with 
	$\Gnp{\abs{S}}{p}$ (see \autoref{rem:indep_is_gnp}).
	According to \autoref{thm:chromatic_constant_p_Gnp},  \autoref{alg:chrom_constant_p_Gnp} uses more than 
	$\left(1+\frac{\epsilon}{2}\right) \frac{\log(b) n}{2\log(n)}$ colors to color $\Gnp{n}{p}$ with
	probability at most $ \exp\left(-\frac{c n^2}{\log(n)^8}\right)$ for some constant $c>0$.
	Note that the probability $p$ is constant by the assumption of the theorem and thus the term $cp^7$ 
	from \autoref{thm:chromatic_constant_p_Gnp} is constant here.
	Choose $\epsilon^\prime>0$ such that $1+\epsilon^\prime= \left(1+\frac{\epsilon}{2}\right)^{\frac{1}{2}}.$
	Combining these facts we get for each color class $S\in\Clarge{g}$ 
	\begin{align}\label{ineq:each_color_class_p_constant}
		&\Pof{\xi\left(\pert{H}{p}[S]\right) \geq \left(1+\frac{\epsilon}{2}\right)^{\frac{1}{2}}
			\frac{\log(b)\abs{S}}{2\log(\abs{S})}}\notag\\
		= &\Pof{\xi\left(\Gnp{\abs{S}}{p} \right) \geq \left(1+\frac{\epsilon}{2}\right)^{\frac{1}{2}} 
			\frac{\log(b)\abs{S}}{2\log(\abs{S})}}\notag\\
		= &\Pof{\xi\left(\Gnp{\abs{S}}{p} \right) \geq \left(1+\epsilon^\prime\right)
			\frac{\log(b)\abs{S}}{2\log(\abs{S})}}\notag\\
		\leq & \exp\left( - \frac{c \abs{S}^2}{\log(\abs{S})^8}\right)
	\end{align}
	for some constant $c>0$.
	
	We will use
	\begin{align}\label{ineq:const_approx_beta_in_loglog}
		\chromaticOf{H} = \frac{n}{\beta(n)}\leq \frac{n}{\log(n)^{\frac{1}{2}}}
		\text{ and thus }
		\log(\beta(n)) \geq \log\left(\log(n)^{\frac{1}{2}}\right) = \frac{1}{2} \log(\log(n)).
	\end{align}
	
	And continue 
	\begin{align*}
		&\Pof{\sum_{S\in\Clarge{g}} \xi\left(\pert{H}{p} [S]\right) 
			\geq \left(1+ \frac{\epsilon}{2}\right) \cdot \frac{\log(b) n}{2(\log(n) -\log(\chromaticOf{H}))}}\\
		=	&\Pof{\sum_{S\in\Clarge{g}} \xi\left(\pert{H}{p} [S]\right) 
			\geq \left(1+ \frac{\epsilon}{2}\right) \cdot \frac{\log(b) n}{2\log(\beta(n))}}\\
		\leq&\Pof{\sum_{S\in\Clarge{g}} \xi\left(\pert{H}{p} [S]\right) 
			\geq {\left(1+ \frac{\epsilon}{2}\right)}^{\frac{1}{2}} \cdot \frac{\log(b) \sum_{S\in\Clarge{g}}\abs{S}}{2\log(g(n))}}
		\tag{since $\displaystyle\sum_{S\in\Clarge{g}}\abs{S}\leq n$}\\
		=	&\Pof{\sum_{S\in\Clarge{g}} \xi\left(\pert{H}{p} [S]\right) 
			\geq \sum_{S\in\Clarge{g}} {\left(1+ \frac{\epsilon}{2}\right)}^{\frac{1}{2}} \cdot 
			\frac{\log(b) \abs{S}}{2\log(g(n))}}\\
		\leq&\Pof{\sum_{S\in\Clarge{g}} \xi\left(\pert{H}{p} [S]\right) 
			\geq \sum_{S\in\Clarge{g}} {\left(1+ \frac{\epsilon}{2}\right)}^{\frac{1}{2}} \cdot 
			\frac{\log(b) \abs{S}}{2\log(\abs{S})}}\tag{since $g(n)\geq \abs{S}$}\\
		\leq&\sum_{S\in\Clarge{g}} \Pof{\xi\left(\pert{H}{p} [S]\right) 
			\geq {\left(1+ \frac{\epsilon}{2}\right)}^{\frac{1}{2}} \cdot \frac{\log(b) \abs{S}}{2\log(\abs{S})}}
		\tag{union bound} \\
		\leq&\sum_{S\in\Clarge{g}} \exp\left(-\frac{c \abs{S}^2}{{\log(\abs{S})}^8}\right)
		\tag{by (\ref{ineq:each_color_class_p_constant})}\\
		\leq&\sum_{S\in\Clarge{g}} \exp\left(-\frac{c g(n)^2}{{\log(g(n))}^8}\right)
		\tag{since $\frac{x^2}{\log(x)^8}$ is an increasing funnction}\\
		\leq&\chromaticOf{H} \cdot \exp\left(-\frac{c\cdot g(n)^2}{{\log(g(n))}^8}\right)\tag{since $\abs{\Clarge{g}}\leq\chromaticOf{H}$}\\
		\leq&\frac{n}{\beta(n)}\cdot \exp\left(-\frac{c\cdot g(n)^2}{{\log(g(n))}^8}\right)\\
		=	&\exp\left(\log(n) -\log(\beta(n)) -
		\frac{c g(n)^2}{\log(g(n))^8}\right)\\
		=	&\exp\left(\log(n) - \log(\beta(n)) - 
		\frac{c^\prime \beta(n)^{2(1+\frac{\epsilon}{2})^{-\frac{1}{2}}}}{\log(\beta(n))^8}\right)
		\tag{for some $c^\prime >0$, since $g(n) 
			= \beta(n)^{\left(1+\frac{\epsilon}{2}\right)^{-\frac{1}{2}} }$}\\
		\leq&\exp\left(\log(n) - \gamma\log(\log(n))- \frac{c^{\prime \prime} \log(n)^{\gamma2(1+\frac{\epsilon}{2})^{-\frac{1}{2}}}}{\log(\log(n))^8}\right)
		\tag{for some $c''>0$ by \ref{ineq:const_approx_beta_in_loglog}}\\
		=	&o(1)\tag{since $\gamma\geq \frac{1}{2}$}.
	\end{align*}
\end{proof}

We proceed to color an augmented graph, where  $p(n)~=~ n^{-\theta}$
for some $\theta\in\left(0,\frac{1}{3}\right)$.
We need the greedy algorithm to color the small color classes. 	

\begin{algorithm}[H]
	\caption{Greedy Color}\label{alg:chromatic_greedy}
	\SetAlgoLined
	\KwIn{A random graph $\Gnp{n}{p}$}
	\KwOut{A coloring of the vertices of $\Gnp{n}{p}$}
	\While{There is a set of uncolored vertices}{
		Choose a maximally independent set $I$ of uncolored vertices\;
		Color $I$ with a new color\;
	}
\end{algorithm}

We also give the following minor improvement of the result for the greedy color algorithm.
\begin{theorem}\label{thm:chromatic_greedy}
	Let $n^{-\tau}\leq p \leq n^{-\theta}$ for some arbitrary, 
	but fixed $\tau,\theta \in \left(0,\frac{1}{3}\right)$.
	Further let $\chi_g\left(\Gnp{n}{p}\right)$ be the number of colors used by Greedy Color.
	Then 
	\begin{align*}
		\Pof{\chi_g\left(\Gnp{n}{p}\right) \geq (1+\epsilon) \frac{n \log(b)}{\log(np)}} \leq  
		\exp\left(- (1+o(1)) n^3\right).
	\end{align*}
\end{theorem}

Furthermore we use the following algorithm of Bollobás \cite{bollobas_chromatic_1988}.

\begin{algorithm}[H]
	\caption{}\label{alg:chrom_small_p_Gnp}
	\SetAlgoLined
	\KwIn{A random graph $\Gnp{n}{p}$}
	\KwOut{A coloring of the vertices of $\Gnp{n}{p}$}
	\While{At least $\frac{n}{{\log(np)}}$ vertices remain uncolored and if there is an uncolored independent set $I$ of size $2 \log_b(np) - 4 \log_b(\log_b(np))$}{Color $I$ with a new color\;}
	color the remaining graph using \autoref{alg:chromatic_greedy}\;
\end{algorithm}

We also give the following minor improvement of a result of Bollobás.
\begin{theorem}\label{thm:chromatic_small_p_Gnp}
	Let $\theta\in\left(0,\frac{1}{3}\right)$, $\delta>0 $ and $n^{-\frac{1}{3} + \delta} \leq p(n) \leq n^{-\theta}$. 
	For sufficiently large $n\in \N$ there is a constant $c>0$ such that
	\begin{align*}
		\Pof{\chromaticOf{\Gnp{n}{p}} \geq (1+ \epsilon) \frac{np}{2\log(np)}} \leq 
		\exp\left(-  \left(\frac{c n^2 p^3}{\log(n)^6}\right)\right).
	\end{align*}
	Furthermore the probability that \autoref{alg:chrom_small_p_Gnp} fails to construct a coloring with
	at most $(1+ \epsilon) \frac{np}{2\log(np)}$ colors is at most 
	$\exp\left(-  \left(\frac{c n^2 p^3}{\log(n)^6}\right)\right).$
\end{theorem}

Based on these results we use the following algorithm.

\begin{algorithm}[H]
	\caption{A coloring of $\pert{H}{p}$}\label{alg:chrom_small_p_pertHp}
	\KwIn{A constant $\theta \in \left(0,\frac{1}{3}\right)$, an edge probability $p$ with $p=p(n)= n^{-\theta}$, a deterministic graph $H$ with $\chromaticOf{H}\leq \frac{n}{\beta(n)}$ where $\beta(n)\geq n^{\frac{3\theta}{2} + \delta}$ for some $\delta>0$ arbitrary, but fixed and a perturbed graph $G= \pert{H}{p}$.
		Further an $\epsilon>0$ arbitrary small but fixed.}
	\KwOut{A proper coloring of $G$ using less than  $(1+\epsilon) \cdot \frac{n \log(b)}{2 (\log(n) - \log(\chromaticOf{H}))}$ colors.}
	
	Construct an optimal coloring $ \chi^\prime$ of $H$\;
	In each color class $S$ of $\chi^\prime$ with $\abs{S} \geq 
	\frac{\beta(n)}{\log(n)}=: g(n)$ construct a proper coloring using  
	\autoref{alg:chrom_small_p_Gnp}\;
	Color each remaining color class of $H$ with the greedy algorithm \autoref{alg:chromatic_greedy}\;
\end{algorithm}

\begin{theorem}
	Let $H=([n],E)$ be a deterministic graph with $\chromaticOf{H} = 
	\frac{n}{\beta(n)}$ for some function $\beta : \N \rightarrow \Q$ such that $\beta(n)
	\rightarrow \infty$ for $n \rightarrow \infty$.
	Let further be $p =p(n) = n^{-\theta}$ for some constant 
	$\theta \in (0,\frac{1}{3})$, let $b=\frac{1}{1-p}$ and $\epsilon>0$. 
	If $\beta(n) \geq n^{3\theta + \epsilon}$,
	\autoref{alg:chrom_small_p_pertHp} \aas constructs a proper coloring of 
	$\pert{H}{p}$ with at most
	$(1+\epsilon) \frac{n \log(b)}{2(\log(n) - \log(\chromaticOf{H}))}$ colors.
\end{theorem}
\begin{proof}
	Let $\xi(\pert{H}{p})$ be the number of colors used by \autoref{alg:chrom_small_p_pertHp}.
	Since this algorithm yields a proper coloring, it is left to prove that 
	$$\xi(\pert{H}{p}) \leq (1+ \epsilon) \frac{n \log(b)}{2(\log(n) - \log(\chromaticOf{H}))} \quad\aas{} $$ 
	With $g=g(n)=\frac{\beta(n)}{\log(n)}$ we  consider $\Csmall{g}$ as well as $\Clarge{g}$.
	We split our proof into two parts.
	First we prove that all "small" color classes are colored using at most 
	$\frac{\epsilon}{2}\frac{n \log(b)}{2(\log(n)-\log(\chromaticOf{H}))}$ colors.
	Next we will prove that all "large" color classes are colored with at most 
	$\left(1+ \frac{\epsilon}{2}\right) \frac{n \log(b)}{2(\log(n) - \log(\chromaticOf{H}))}$ colors.
	
	{\it Claim 1:} $\xi(\pert{H}{p}\left[{\bigcup_{S\in \Csmall{g} }S}\right]) 
	<\frac{\epsilon}{2}\frac{n \log(b)}{2(\log(n)-\log(\chromaticOf{H}))}$ \aas\\
	
	{\it Proof of Claim 1:} Let $\chi_g(\Gnp{n}{p})$ be the number of colors that are used when coloring $\Gnp{n}{p}$
	with the greedy \autoref{alg:chromatic_greedy}.
	Since for each $S\in \Csmall{g}$ the graph $\pert{H}{p}[S]$
	has the same distribution as the random graph $\Gnp{\abs{S}}{p(n)}$
	by \autoref{rem:indep_is_gnp}.
	We may consider the random graph $\Gnp{g(n)}{p(n)}$ 
	in which each set of $\abs{S}$ vertices induces a 
	$\Gnp{\abs{S}}{p(n)}$.
	Thus the greedy algorithm uses at least as many colors to color
	$\Gnp{g(n)}{p(n)}$  as for $\pert{H}{p}[S].$
	Since \autoref{alg:chrom_small_p_pertHp} uses the greedy \autoref{alg:chromatic_greedy} 
	to color $S,$ we get $\xi(\pert{H}{p}[S]) = \chi_g(\Gnp{\abs{S}}{p(n)})$ 
	where $\chi_g(\Gnp{\abs{S}}{p(n)})$ is the number of colors used by 
	\autoref{alg:chromatic_greedy} to color $\Gnp{\abs{S}}{p(n)}$.\\
	Thus for all $q \in \R$ we have,
	\begin{align}\label{ineq:greedy_sub}
		\Pof{\xi(\pert{H}{p(n)}[S]) \geq q} 
		= \Pof{\chi_g(\Gnp{\abs{S}}{p(n)})\geq q} 
		\leq \Pof{\chi_g(\Gnp{g(n)}{p(n)})\geq q }.
	\end{align}
	Now
	\begin{align*}
		&\Pof{\xi\left(\pert{H}{p}\left[\bigcup_{S\in \Csmall{g}}S\right]\right) \geq \frac{\epsilon}{2} 
			\cdot\frac{n \log(b)}{2(\underbrace{\log(n)-\log(\chromaticOf{H}}_{=\log(\beta(n))}))}}\\
		\leq& \Pof{\sum_{S\in \Csmall{g}}\xi(\pert{H}{p}[S]) \geq \frac{\epsilon}{2} 
			\cdot\frac{n \log(b)}{2\log(\beta(n))}}\\
		\leq& \Pof{\sum_{S\in \Csmall{g}}\xi(\pert{H}{p}[S]) \geq \frac{\epsilon}{2} 
			\cdot\frac{\sum_{S\in \Csmall{g}} \beta(n) \log(b)}{2 \log(\beta(n))}}
		\tag{$n\geq \sum_{S\in \Csmall{g}} \beta(n)$}\\
		\leq& \sum_{S\in \Csmall{g}}\Pof{\xi(\pert{H}{p}[S]) \geq  
			\frac{\epsilon}{2} \cdot\frac{\beta(n) \log(b)}{2 \log(\beta(n))}}
		\tag{by \autoref{lem:union_bound}}\\
		\leq& \sum_{S\in \Csmall{g}}\Pof{\chi_g(\Gnp{g(n)}{p}) \geq 
			\frac{\epsilon}{2} \cdot\frac{\beta(n) \log(b)}{2 \log(\beta(n))}}
		\tag{by inequality \ref{ineq:greedy_sub}}\\
		\leq& \sum_{S\in \Csmall{g}}\Pof{\chi_g(\Gnp{g(n)}{p}) \geq (1+\epsilon) 
			\cdot\frac{g(n) \log(b)}{2 \log(g(n))}}\tag{$g(n) = \frac{\beta(n)}{\log(n)}$}\\
		\leq& n \exp\left(- c g(n)^3\right) \leq \exp(-cn^2) = o(1). \tag{by \autoref{thm:chromatic_greedy}}
	\end{align*}
	Now we show that the large color classes are colored with not too many colors.
	
	{\it Claim 2:} $\xi(\pert{H}{p}\left[{\bigcup_{S\in \Clarge{g} }S}\right]) < 
	\left(1+ \frac{\epsilon}{2}\right) \frac{n \log(b)}{2(\log(np) - \log(\chromaticOf{H}))}$ \aas\\
	
	{\it Proof of Claim 2:}
	Note that for all $n\geq n(\epsilon)$ for some $n(\epsilon) \in \N,$ we have
	\begin{align}\label{ineq:log_g_n_log_beta_n}
		\left(1+\epsilon\right)^{\frac{1}{2}}\log(g(n) p) \geq \log(\beta(n)p) 
	\end{align}
	We have
	\begin{align*}
		&\Pof{\sum_{S\in\Clarge{g}} \xi\left(\pert{H}{p} [S]\right) \geq 
			\left(1+ \frac{\epsilon}{2}\right) \cdot \frac{\log(b) n}{2(\log(np) -\log(\chromaticOf{H}))}}\\
		=	&\Pof{\sum_{S\in\Clarge{g}} \xi\left(\pert{H}{p} [S]\right) 
			\geq \left(1+ \frac{\epsilon}{2}\right) \cdot \frac{\log(b) n}{2\log(\beta(n)p)}}
		\tag{$\beta(n) = n/\chromaticOf{H}$}\\
		\leq&\Pof{\sum_{S\in\Clarge{g}} \xi\left(\pert{H}{p} [S]\right) 
			\geq {\left(1+ \frac{\epsilon}{2}\right)} \cdot \frac{\log(b)
				\sum_{S\in\Clarge{g}}\abs{S}}{2\log(\beta(n)p)}}
		\tag{$n\geq \sum_{S\in\Clarge{g}} \abs{S}$}\\
		\leq&\Pof{\sum_{S\in\Clarge{g}} \xi\left(\pert{H}{p} [S]\right) 
			\geq {\left(1+ \frac{\epsilon}{2}\right)}^{\frac{1}{2}} \cdot \frac{\log(b)
				\sum_{S\in\Clarge{g}}\abs{S}}{2\log(g(n)p)}}
		\tag{by inequality \ref{ineq:log_g_n_log_beta_n}}\\
		\leq&\Pof{\sum_{S\in\Clarge{g}} \xi\left(\pert{H}{p} [S]\right) \geq
			\sum_{S\in\Clarge{g}} {\left(1+ \frac{\epsilon}{2}\right)}^{\frac{1}{2}} \cdot 
			\frac{\log(b) \abs{S}}{2\log(\abs{S}p)}}
		\tag{$\abs{S}\geq g(n)$}\\
		\leq&\sum_{S\in\Clarge{g}} \Pof{\xi\left(\pert{H}{p} [S]\right) \geq 
			{\left(1+ \frac{\epsilon}{2}\right)}^{\frac{1}{2}} \cdot 
			\frac{\log(b) \abs{S}}{2\log(\abs{S}p)}}
		\tag{by \autoref{lem:union_bound}}\\
	\end{align*}
	Note that there is some $\delta\in\left(0,\frac{1}{3}\right)$ such that for all $S\in\Clarge{g}$
	\begin{align}
		p(n) \geq \abs{S}^{-\frac{1}{3}+\delta}.
	\end{align}
	This can be proved as follows
	\begin{align}
		p(n) = n^{-\theta} = \left(n^{3\theta}\right)^{-\frac{1}{3\theta} \theta}
		\geq \beta(n)^{-\frac{1}{3}+\delta} \geq \abs{S}^{-\frac{1}{3}+\delta}.
	\end{align}
	By \autoref{rem:indep_is_gnp} each large color class $S\in \Clarge{g}$ of $H$ induces a random graph\newline
	$\pert{H}{p}~[S]~=~\Gnp{|S|}{p}$.
	Thus \autoref{alg:chrom_small_p_Gnp} can be used to color each color class in $\Clarge{g}$.
	Thus we can apply \autoref{thm:chromatic_small_p_Gnp} with $\epsilon' >0,$ 
	$1+\epsilon'= {\left(1+ \frac{\epsilon}{2}\right)}^{\frac{1}{2}},$
	\begin{align*}
		&\Pof{\xi\left(\pert{H}{p} [S]\right) 
			\geq {\left(1+ \frac{\epsilon}{2}\right)}^{\frac{1}{2}} \cdot \frac{\log(b) 
				\abs{S}}{2\log(\abs{S}p)}} \\
		=&\Pof{\xi\left(\Gnp{|S|}{p}\right) 
			\geq {\left(1+ \frac{\epsilon}{2}\right)}^{\frac{1}{2}} \cdot \frac{\log(b) 
				\abs{S}}{2\log(\abs{S}p)}} \\
		=&\Pof{\xi\left(\Gnp{|S|}{p}\right) 
			\geq {\left(1+ \epsilon'\right)} \cdot \frac{\log(b) 
				\abs{S}}{2\log(\abs{S}p)}} \\
		\leq &\exp\left(- \frac{c\abs{S}^2p^3}{\log(\abs{S})^6}\right)
		\tag{ for some $c>0$ by \autoref{thm:chromatic_small_p_Gnp}}\\
		\leq &\exp\left(-\frac{cg(n)^2p^3}{\log(g(n))^6}\right)
		\tag{since $\frac{x^2}{\log(x)^6}$ is monotonously increasing}.
	\end{align*}
	By our definition $\chromaticOf{H} = \frac{n}{\beta(n)},$ so trivially 
	$\abs{\Clarge{g}} \leq \frac{n}{\beta(n)}.$ 
	Thus
	\begin{align*}
		&\Pof{\sum_{S\in\Clarge{g}} \xi\left(\pert{H}{p} [S]\right) \geq 
			\left(1+ \frac{\epsilon}{2}\right) \cdot \frac{\log(b) n}
			{2(\log(np) -\log(\chromaticOf{H}))}}\\
		\leq &\frac{n}{\beta(n)} \exp\left(-\frac{cg(n)^2p^3}
		{\log(g(n))^6}\right)\\
		=	 &\exp\left(\log(n) -\log(\beta(n)) -
		\frac{cg(n)^2p^3}{\log(g(n))^6}\right)\\
		= 	 &\exp\left(\log(n) -\log(\beta(n)) - \frac{c\beta(n)^2p^3}
		{\log(n)^2\log(g(n))^6}\right)
		\tag{$g(n) = \frac{\beta(n)}{\log(n)}$}\\
		\leq &\exp\left(\log(n) -\log(\beta(n)) -
		\frac{c\beta(n)^2p^3}{\log(n)^8}\right)
		\tag{$\log(n) \geq\log(g(n))$}\\
		\leq &\exp\left(\log(n) -\log(\beta(n)) -
		\frac{cn^{6\theta+2\epsilon}p^3}{\log(n)^8}\right)
		\tag{$\beta(n) \geq n^{3\theta}$}\\
		=	&\exp\left(\log(n) -\log(\beta(n)) -
		\frac{cn^{6\theta+2\epsilon -3\theta}}{\log(n)^8}\right)
		\tag{$p(n) = n^{-\theta}$}\\
		=	&\exp\left(\log(n) -\log(\beta(n)) -
		\frac{cn^{3\theta+2\epsilon}}{\log(n)^8}\right)\\
		= 	 &o(1).
	\end{align*} 
	
	We combine claim 1 and 2 conclude the proof: 
	\begin{align*}
		&\Pof{\xi\left(\pert{H}{p}\right) \geq \left(1+\epsilon\right) \frac{n \log(n)}{2(\log(np) - \log(\chromaticOf{H}))}}\\
		\leq &\Pof{\xi\left(\pert{H}{p}\left[\bigcup_{S\in \Clarge{g} }S\right]\right) < \left(1+ \frac{\epsilon}{2}\right) \frac{n \log(b)}{2(\log(np) - \log(\chromaticOf{H}))}}\\
		&+ \Pof{\xi\left(\pert{H}{p}\left[{\bigcup_{S\in \Csmall{g} }S}\right]\right) <\frac{\epsilon}{2}\frac{n \log(b)}{2(\log(np)-\log(\chromaticOf{H}))}}\\
		=	 &o(1).
	\end{align*}
\end{proof}

\section{Conclusion and Open Questions}
\label{sec:conclusion_chromatic}
In this paper we first proved that for certain combinations of host graphs $H$ with bounded chromatic
number, and constant edge probabilities $p$ or edge probabilities $p=p(n)\geq n^{-\frac{1}{3}+\delta}$
for some $\delta>0$, the chromatic number of the augmented graph $\pert{H}{p}$
is $\aas$ bounded from above by
\begin{align*}
	\frac{n \log(b)}{2 (\log(n) -\log(\chromaticOf{H})) }
\end{align*}
(\autoref{thm:chromatic_constant_p_aas} and \autoref{thm:chromatic_small_p_aas}).
Furthermore it was shown that this upper bound is asymptotically tight for host graphs that
possess only few independent sets of a certain size (\autoref{thm:lower_bound_aas}).\\
In addition the bound is asymptotically tight for host graphs $H$ with finite chromatic number, 
as in this case for each $\epsilon>0.$
\begin{align*}
	\chromaticOf{\pert{H}{p}} 
	\geq \chromaticOf{\Gnp{n}{p}} 
	\geq \left(1-\frac{\epsilon}{2}\right) \frac{n \log(b)}{2\log(n)} 
	\geq (1-\epsilon) \frac{n \log(b)}{2(\log(n)-\log(\chromaticOf{H}))}.
\end{align*} 
However the upper bound cannot be tight for each host graph.
Consider for example a host graph chosen uniformly at random from all graphs with $n$ vertices.  
This is equivalent to considering $H$ to be a $\Gnp{n}{\frac{1}{2}}$ random graph.
Thus  $\chromaticOf{H}~=~(1~+~o~(1))~\frac{n\log(2)}{2\log(n)}$ $\aas$ and 
$H$ satisfies the conditions of \autoref{thm:chromatic_constant_p_aas}.
So the upper bound given by \autoref{thm:chromatic_constant_p_aas} yields $\aas$
\begin{align}\label{ineq:conc_weak}
	\chromaticOf{\pert{H}{p}} 
	\leq (1+\epsilon)\frac{n\log(b)}{2\log(\log(n))}.
\end{align}
In this case $\pert{H}{p}$ can be considered as a random graph $\Gnp{n}{\hat{p}}$ where
$\hat{p} = \frac{1}{2} (p+1),$ thus with $\hat{b} = \frac{1}{1-\hat{b}}$
we have $\aas$
\begin{align}\label{ineq:conc_strong}
	\chromaticOf{\pert{H}{p}}
	= \chromaticOf{\Gnp{n}{\hat{p}}}
	\leq (1+o(1)) \frac{n \log\left(\hat{b}\right)}{2\log(n)}.
\end{align}
The upper bound on $\chromaticOf{\pert{H}{p}}$ in \ref{ineq:conc_weak} is smaller than the upper
bound in \ref{ineq:conc_strong} by a factor of $c\frac{\log(n)}{\log(\log(n))}.$
We ask the following natural question:\\
{\bf Question 1:} Since now an asymptotically tight upper bound for the chromatic number of augmented
graphs is known, it would be interesting whether there is a better lower bound than the trivial bound
\begin{align*}
	\chromaticOf{\pert{H}{p}} 
	\geq \max(\chromaticOf{H}, \chromaticOf{\Gnp{n}{p}})
	= \max\left(\chromaticOf{H}, \frac{n \log(b)}{2\log(n)}\right)?
\end{align*}

The chromatic number of $\Gnp{n}{p}$ is known for a much wider range of $p$.
Furthermore no result was given for the case $\chromaticOf{H} = \Theta(n).$
We may ask \\
{\bf Question 2:} Since an upper bound for the chromatic number of $\pert{H}{p}$ for the cases
$p(n) \geq n^{\frac{1}{3}+\delta}$ and $\chromaticOf{H} = o(n)$ has been established 
in this paper, can it be extended to 
include the case $p(n) = o(1)$ and host graphs with $\chromaticOf{H} = \Theta(n)$?

The strategy of divide and conquer that was applied in this paper was shown to be useful 
in our context.
However it appears to be rather ill suited to study most other classical graph invariants such as 
the independence number.
Thus we may ask\\
{\bf Question 3:} Which graph invariants of randomly augmented graphs are well suited to be 
studied by a divide and conquer strategy?

\section{Acknowledgments}
We thank Joel Spencer for useful comments during his 2019 visit at Kiel University and we thank Lutz
Warnke for bringing to our attention a recent paper of published by Surya and Warnke proving 
concentration inequalities for $\chromaticOf{\Gnp{n}{p}}$ where $p= p(n) = n^{-\alpha}$ which has a 
high potential to 
be of help in tackling this case for randomly augmented graphs \cite{surya_concentration_2024}.
\newpage
\bibliographystyle{alpha}
\bibliography{sources}

\newpage

\appendix

\section{Known results for \texorpdfstring{$\chromaticOf{\Gnp{n}{p}}$}{Gn,p} and some consequences}
\label{sec:Known_results}
\label{chap:chromatic_Gnp}

In this appendix we will prove upper bounds of $\chromaticOf{\Gnp{n}{p}}$ that hold \aas{}
To tackle this problem we will split $p$ into multiple ranges to control the vertices that 
are not included in large independent sets.
We start by proving the basic results \autoref{lem:divide_and_color} and \autoref{lem:union_bound}.

\begin{proof}[Proof of \autoref{lem:divide_and_color}]
	For each $S\in\Call$ color the graph with $\chromaticOf{H[S]}$ colors that are not yet used.
	This is a proper coloring of $G\cup H$ since no edges in $G$ exist between vertices of the same color, 
	as $\Call$ is a set of independent sets.
	Furthermore by construction there are no edges in $H$ between vertices of the same color, 
	since no color is used for vertices contained in different members of $\Call$.
	This proves the inequality.
	Let $G$ be a complete $k$-partite graph and $\Call$ is a $k$-partition of the vertices of $V$ as above,
	the coloring is even optimal, since the $H[S]$ are already optimally colored and for each pair of members of $\Call$,
	every edge between them already exists. 
	Thus every other proper coloring of the graph uses at least as many colors as the coloring constructed here.
\end{proof}

\begin{proof}[Proof of \autoref{lem:union_bound}]
	Let 
	$\tilde{\Omega} :=  \left\{\omega \in \Omega : \sum_{i = 1}^k X_i(\omega) > \sum_{i=1}^k A_i \right\}$ 
	and for each $i\in [k]$ let 
	$ \Omega_i := \left\{ \omega \in \Omega : X_i(\omega)  > A_i \right\}.$ 
	If for some $\omega \in \Omega$ there is $X_i(\omega) > A_i$ for all $i\in[k],$
	then $\sum_{i=1}^{k} X_i > \sum_{i=1}^{k} A_i.$
	So, $\tilde{\Omega} \supseteq \bigcap_{i\in[k]} \Omega_i$ and 
	$\tilde{\Omega}^c \subseteq (\bigcap_{i\in[k]} \Omega_i)^c = \bigcup_{i\in [k]} \Omega_i^c.$
	This implies
	\begin{align*}
		\Pof{\sum_{i=1}^{k} X_i \leq \sum_{i=1}^k A_i} 
		= \Pof{\tilde{\Omega}^c} 
		\leq \Pof{\bigcup_{i\in [k]} \Omega_i^c}
		\leq \sum_{i=1}^{k}\Pof{\Omega_i^c}
		= \sum_{i=1}^{k} \Pof{X_i\leq A_i}.
	\end{align*}
\end{proof}

The following results and proofs are concerned with the chromatic number of $\Gnp{n}{p}$.
Variations of these results are known, but a new result by Krivelevich, Sudakov, van Vu and Wormald 
\cite{krivelevich_probability_2003} allows a slight strengthening of the probability bounds
with the same techniques.

First we prove the following result which was stated but not proved in \cite{krivelevich_probability_2003}.
Please note that proofs of variations of this result are often left as exercises to the reader.
For completeness sake the proof shall be included here.
\begin{lemma}
	For each $n\in\N$ and $p= n^{-\theta}$ for some $\theta \in (0,1)$ let 
	\begin{align*}
		k_0(n) := \max \left\lbrace k \in \N : \binom{n}{k} {(1-p)}^{\binom{k}{2} } > n^4\right\rbrace.
	\end{align*}
	Then $k_0 \sim 2\log_b(np)$ and more precisely 
	$ 2 \log_b(np) - 4\log_b(\log(np)) \leq k_0 \leq 2 \log_b(np)$.
\end{lemma}
\begin{proof}
	Let $k_1 = 2 \log_b(np)$.
	Then 
	\begin{align*}
		&\binom{n}{k_1} {(1-p)}^{\binom{k_1}{2}} 
		\leq \left(\frac{n e}{k_1 (1-p)^{\frac{1}{2}}} (1-p)^\frac{k_1}{2}\right)^{k_1}\\
		&= \left(\frac{n \log(b)e}{2\log(np) (1-p)^{\frac{1}{2}}}\cdot b^\frac{-2\log_b(np)}{2}\right)^{k_1}\\
		&= \left(\frac{ e}{2\log(np) (1-p)^{\frac{1}{2}}}  \cdot\frac{n \log(b)}{np}\right)^{k_1}\\
		&= \left(\frac{ (1+\oOf{1})e}{2\log(np) (1-p)^{\frac{1}{2}}} \right)^{k_1}\tag{by \ref{eq:log_properties}}
		=\oOf{1}.
	\end{align*}
	
	Let further $k_2 = 2 \log_b(np) - 4\log_b(\log(np))$. Then
	\begin{align*}
		&\binom{n}{k_2} {(1-p)}^{\binom{k_2}{2}} \\
		\geq& \left(\frac{n}{k_2 (1-p)^{\frac{1}{2}}} (1-p)^\frac{k_2}{2}\right)^{k_2}\\
		=& \left(\frac{n \log(b)}{2\log(np)(1-p)^{\frac{1}{2}}}\cdot b^\frac{-2\log_b(np)+ 4\log_b(\log(np))}{2}\right)^{k_2}\\
		=& \left(\frac{ 1}{2\log(np) (1-p)^{\frac{1}{2}}}  \cdot\frac{n \log(b)}{np} \cdot (\log(np))^2\right)^{k_2}\\
		=& \left(\frac{ (1+\oOf{1})\log(np) }{2(1-p)^{\frac{1}{2}}} \right)^{k_2}&\tag*{(by (\ref{eq:log_properties}))}\\
		\geq&c^{k_2} \log(np)^{k_2} &\tag*{(for some constant $c>0$)}\\
		=&\exp\left(k_2(\log(c) +\log(\log(np)))\right)\\
		\geq& \exp\left( \log_b(np)(\log(c) +\log(\log(np)))\right) &\tag*{(since $k_2\geq \log_b(n)$)}\\ 
		= & (np)^{\log(b)^{-1} (\log(c) +\log(\log(np)))}&\tag*{(by (\ref{eq:log_properties}))}\\
		>& n^4.\tag*{\qedhere}
	\end{align*}
	Thus we have $k_2\leq k_0 \leq k_1$ and $k_1,k_2$ are both asymptotically equal to $2 \log_b(np).$
\end{proof}

It was shown in \cite{krivelevich_probability_2003} that the independence number of $\Gnp{n}{p}$ is 
at least $k_0$ asymptotically almost surely:
\begin{theorem}\label{thm:independence_all_p_aas}
	Let $n^{-\frac{2}{5} + \delta} \ll p(n) \leq 1-\delta$ for some $\delta >0$.
	Then for each $\epsilon>0$ there exists $n=n(\epsilon)\in\N$ so that for all $n\geq n(\epsilon)$
	there are constants $c^\prime , c>0$ such that
	\begin{align*}
		\Pof{ \alpha(\Gnp{n}{p})< k_0} \leq \exp\left( - \frac{c n^2}{{k_0(n)}^4 p}\right) \leq \exp\left( - \frac{c^\prime n^2 p^3}{{\log(n)}^4}\right).
	\end{align*}
\end{theorem}
This has been an improvement over the probability bound of the well-known result of Matula \cite{matula_largest_1970}.

The following \autoref{lem:no_independence_Gnp_constant_p} and \autoref{thm:chromatic_constant_p_Gnp} will be proved with arguments from \cite{bollobas_chromatic_1988} and \cite{frieze_introduction_2016} combining \autoref{alg:chrom_constant_p_Gnp} and \autoref{thm:independence_all_p_aas}.
The difference to  the former two results is that they used other results on the independence number of
$\Gnp{n}{p}$.

The following result constitutes a minor improvement of the result of Bollobás \cite{bollobas_chromatic_1988}
for the upper bound of the chromatic number of $\Gnp{n}{p}$ for constant $p$.
\begin{lemma}\label{lem:no_independence_Gnp_constant_p}
	Let $\delta>0$ and $(1-\delta) \geq p\gg n^{-\theta}$ for each constant $\theta \in (0,\frac{1}{3})$.
	Then the probability that there exists a vertex set $V_1 \subseteq V$ set of size $\nu(n) = \frac{n}{{\log_b(np)}^2}$
	such that $V_1$ does not contain an independent set of size $k_0(\nu) = 2 \log_b(np)-4\log_b(\log_b(np))$
	is super exponentially small, more exactly
	\begin{align*}
		\Pof{\exists S\subseteq [n] : \abs{S} = \nu \land \alpha(\Gnp{n}{p}[S]) < k_0(\nu)}\leq 
		\exp\left(- \frac{c n^2}{\log(n)^8}\right).
	\end{align*}
\end{lemma}
\begin{proof}
	Note that with \autoref{eq:log_properties} we have 
	\begin{align} \label{eq:known_1}
		k_0(\nu) \leq 2\log_b(np)=2\frac{\log(np)}{\log(b)}
	\end{align}
	By \autoref{thm:independence_all_p_aas}, for sufficiently large $n\in\N$, there exists a constant $c>0$ such that
	\begin{align*}
		&\Pof{\exists S\subseteq [n] : \abs{S} = \nu \land \alpha(\Gnp{n}{p}[S]) < k_0(\nu)}\\
		\leq&\binom{n}{\nu} \cdot \exp\left( - \frac{c \nu^2}{k_0(\nu)^4 p}\right)  
		&\tag*{(union bound and \autoref{thm:independence_all_p_aas})}\\
		\leq&n^\nu \cdot \exp\left(-\frac{c \nu^2  \log(b)^4}{\log(n)^4 p}\right) 	\\
		=&\exp\left(- \frac{c_1 \nu^2}{\log(n)^4} + \log(n) \nu \right) 
		&\tag*{(with some $c_1>0$ using (\ref{eq:known_1}))}\\
		=&\exp\left(- \nu \left(\frac{c_1 \nu}{\log(n)^4} - \log(n) \right)\right) \\
		=&\exp\left(-\frac{c_1 \nu^2}{\log(n)^4} (1-o(1))\right)  
		&\tag*{$\left(\text{since } \log(n) = o\left( \frac{c_1 \nu}{\log(n)^4}\right)\right)$} \\
		=& \exp\left(-\frac{c_1 n^2 }{\log(n)^4 \log_b(np)^4}(1-o(1))\right)
		&\tag*{$\left(\text{since }\nu=\frac{n}{\log_b(np)}\right)$}\\
		\leq& \exp\left(-\frac{c_1 n^2 \log(b)^4}{\log(n)^8}(1-o(1))\right)
		&\tag*{(using (\autoref{rem:log_properties}))}\\
		\leq&\exp\left(-\frac{c_2 n^2 }{\log(n)^8}\right).    
		&\tag*{(with some constant $c_2>0$)}
	\end{align*}
\end{proof}
We can prove \autoref{thm:chromatic_constant_p_Gnp} 
\begin{proof}[Proof of \autoref{thm:chromatic_constant_p_Gnp}]
	We construct a coloring using \autoref{alg:chrom_constant_p_Gnp}.
	This algorithm yields a proper coloring.
	By \autoref{lem:no_independence_Gnp_constant_p} we can assume that every vertex set $V_1 \subseteq V$ set of size 
	$\nu(n) = \frac{n}{{\log_b(np)}^2}$ contains an independent subset of size 
	$k_0(\nu) = 2 \log_b(np)-4\log_b(\log_b(np))$ with high probability.
	
	Let $\epsilon > 0.$
	\autoref{alg:chrom_constant_p_Gnp} colors each of the independent sets of size $k_0(\nu)$ with a new color, so 
	$ \left\lceil\frac{n-\nu}{k_0(\nu)}\right\rceil \leq (1+\epsilon/2) \frac{n \log(b)}{{2\log(np)}}$ colors are required,
	where the inequality holds for all $n\geq n_1(\epsilon)$ for some 
	$n_1\in \N$.
	Coloring each vertex in the remaining uncolored set with its own color requires at most 
	$\nu = \frac{n}{{\log_b(np)}^2}  \leq \epsilon/2 \cdot \frac{n \log(b)}{{2\log(np)}}$ colors, 
	where the inequality is true for all $n\geq n_2(\epsilon)$ for some
	$n_2(\epsilon) \in \N.$
	So at most $\frac{n-\nu}{k_0(\nu)} + \nu \leq (1+\epsilon) \frac{n \log(b)}{2 \log(np)}$ colors are used.
	
	Since the \autoref{alg:chrom_constant_p_Gnp} yields a proper coloring of $\Gnp{n}{p},$ 
	$\chromaticOf{\Gnp{n}{p}} > \left(1+\epsilon\right)\frac{n \log(b)}{{2\log(np)}}$
	implies that \autoref{alg:chrom_constant_p_Gnp} fails to deliver a coloring with
	at most $ \left(1+\epsilon\right)\frac{n \log(b)}{{2\log(np)}}$ colors.
	
	By the argumentation above this is only the case if there exists $V_1 \subseteq V$
	of size $\nu(n) = \frac{n}{{\log_b(np)}^2}$ that does not contain an independent subset of size 
	$k_0(\nu) = 2 \log_b(np)-4\log_b(\log_b(np))$
	
	Thus, by \autoref{lem:no_independence_Gnp_constant_p}, for all $n~\geq ~n_0 :=~\max\{n_1,n_2\}$
	\begin{align*}
		\Pof{\chromaticOf{\Gnp{n}{p}} \geq (1+\epsilon) \frac{n \log(b)}{2\log(np)}}
		\leq& \Pof{\exists S\subseteq [n] : \abs{S} = \nu \land \alpha(\Gnp{n}{p}[S]) < k_0(\nu)}\\
		\leq& \exp\left(-\frac{cn^2p^7}{\log(n)^8}\right).
	\end{align*}
\end{proof}

\autoref{thm:chromatic_constant_p_Gnp} directly implies \autoref{cor:chromatic_gnp_large_p_expectation}

\begin{proof}[Proof of \autoref{cor:chromatic_gnp_large_p_expectation}]
	Let $\epsilon>0$ and $k= (1+\frac{\epsilon}{2}) \frac{n \log(b)}{2\log(n)}$.
	\begin{align}
		\Eof{\chromaticOf{\Gnp{n}{p}}}
		&= \sum_{i=1}^n i \cdot \Pof{\chromaticOf{\Gnp{n}{p}} =i} \nonumber\\ 
		&\leq \sum_{i=1}^{k} i \cdot \Pof{\chromaticOf{\Gnp{n}{p}} =i} + (n-k) \Pof{\chromaticOf{\Gnp{n}{p}} > k}\nonumber \\
		&\leq k \sum_{i=1}^{k} \Pof{\chromaticOf{\Gnp{n}{p}} =i}+
		n \exp\left(- \frac{c n^2 p^7}{\log(n)^8}\right)\nonumber\\
		&\leq \left(1+\frac{\epsilon}{2}\right) \frac{n\log(b)}{2\log(np)} + 
		n \exp\left(- \frac{c n^2 p^7}{\log(n)^8}\right)\tag{By \autoref{thm:chromatic_constant_p_Gnp} with $\epsilon^\prime = \epsilon/2$}\\
		&\leq \left(1+\frac{\epsilon}{2}\right) \frac{n\log(b)}{2\log(np)} + \frac{\epsilon\cdot n\log(b)}{2 \cdot 2\log(n)}\label{ineq:expectation}\\
		&=  (1+\epsilon) \frac{n\log(b)}{2\log(n)}.\nonumber
	\end{align}
	The inequality in (\ref{ineq:expectation}) holds, since $p(n) > n^{-\frac{1}{7}}$ and thus
	\begin{align*}
		n \exp\left(-\frac{c n^2 p^7}{\log(n)^8}\right)
		\leq n \exp\left(-\frac{c n}{\log(n)^8}\right)
		\leq \frac{\epsilon\cdot n\log(b)}{2 \cdot 2\log(n)}
	\end{align*}
\end{proof}

As seen in the proof of \autoref{thm:chromatic_constant_p_Gnp}, if $p$ is constant, in the final step of
\autoref{alg:chrom_constant_p_Gnp} it is possible to color the remaining vertices with a new color for each such vertex.
This strategy is not successful in calculating a tight bound in case of $p = p(n) \sim n^{-\theta}$.
Here the number of remaining vertices is too large to be handled in such a way.
To be able to analyze colorings of $\Gnp{n}{p}$ for $p = p(n) \sim n^{-\theta}$ where 
$\theta \in \left(0,\frac{1}{3}\right)$ we need \autoref{alg:chromatic_greedy}.

Now we prove \autoref{thm:chromatic_greedy}.

\begin{proof}[Proof of \autoref{thm:chromatic_greedy}]
	We will follow the pattern of the proof in \cite[Theorem 7.9]{frieze_introduction_2016} 
	given for constant $p$ with slight modifications.
	Suppose that in some iteration of \autoref{alg:chromatic_greedy} there are at least $n_0 = \frac{np}{\left(\log_b(n)\right)^2}$ vertices uncolored.
	Let $U$ be the set of uncolored vertices.
	Let $\kappa = \frac{4}{\theta}$ and $k_1= \log_b(np) - \kappa \log_b(\log_b(n))$.
	With the trivial estimate $k_1 \leq \log_b(np)\leq \log_b(np)$ we get
	\begin{align*}
		&\log(k_1) + k_1 \log(n) +2k_1 - \log_b(n)^{\kappa - 2} \\
		\leq &\log(\log_b(np)) + \log_b(np)  \log(n) + 2 \log_b(np) -\log_b(n)^{\kappa-2}\\
		\leq &\log_b(n) (\log(n) + 2 - o(1)) - \log_b(n)^{\kappa-2}\\
		=	 & -\log_b(n)(\log_b(n)^{\kappa - 3} - \log(n)-2+o(1))\\
		= 	 & -\log_b(n)( \log(n)^{\kappa - 3} \log(b)^{-\kappa+3} -\log(n)- 2+o(1))\\
		=	 &-\log_b(n)( \log(n)^{\kappa - 3} (1\pm o(1))^{-\kappa+3} p^{-\kappa+3} -\log(n)- 2+o(1))\\
		& \tag{\text{using $\log(b) = (1\pm \oOf{1}) p$}}\\ 
		= 	 & -\log_b(n)( \log(n)^{\kappa - 3} (1\pm o(1)) p^{-\kappa+3} -\log(n)- 2+o(1))\\
		& \tag{\text{using that $\kappa$ is constant and thus $(1\pm o(1))^{-\kappa+3}= (1\pm \oOf{1})$}}\\
		=	 & -(1\pm o(1))\log_b(n)(\log(n)^{\kappa - 3} p^{-\kappa+3} -\log(n)- 2+o(1))\\
		= 	 & -(1\pm o(1))\log_b(n)p^{-\kappa+3}(\log(n)^{\kappa - 3}  -(\log(n)+ 2+o(1)) p^{\kappa-3})\\
		\leq & -(1\pm o(1))\log_b(n)\log(b)^{-\kappa+3}\cdot(\log(n)^{\kappa - 3}  -
		\underbrace{(\log(n)+ 2) n^{-\tau(\kappa-3)}}_{=o(1)})\\
		& \tag{\text{again with $\log(b) = (1\pm \oOf{1}) p$ and $p\geq n^{-\tau}$}}\\
		= 	& -(1\pm o(1))\log_b(n)\log(b)^{-\kappa+3}\cdot\log(n)^{\kappa - 3}(1-\oOf{1}) \\
		\leq& -(1\pm o(1)) \log_b(n)^{\kappa-2}.
	\end{align*}
	Now let $U$ be a set of cardinality $n_0$, then
	{\allowdisplaybreaks
		\begin{align*}
			&\Pof{\exists S : \abs{S} \leq k_1, S\text{ is a maximally independent set in } U}\\
			\leq&\sum_{S\subseteq U;\abs{S}\leq k_1}\Pof{S\text{ is a maximally independent set in } U}\\
			=	&\sum_{t=1}^{k_1}\sum_{S\subseteq U;\abs{S}=t}\Pof{S\text{ is a maximally independent set in } U}\\
			=&\sum_{t=1}^{k_1}\sum_{S\subseteq U;\abs{S}=t}\Pof{S\text{ is an independent set in } U}\\
			&\cdot	\Pof{\text{for each } u \in (U\setminus S) \text{ there exists } s\in S: \{u,s\}\in E(\Gnp{n_0}{p})}\\
			=	&\sum_{t=1}^{k_1}\sum_{S\subseteq U;\abs{S}=t}\Pof{S\text{ is an independent set in } U}\\
			&\cdot	\prod_{u\in U\setminus S} \Pof{\text{there exists } s\in S: \{u,s\}\in E(\Gnp{n_0}{p})}\\
			=	&\sum_{t=1}^{k_1}\sum_{S\subseteq U;\abs{S}=t}\Pof{S\text{ is an independent set in } U}\\
			&\cdot	\prod_{u\in U\setminus S} (1-\Pof{\text{for all } s\in S: \{u,s\}\notin E(\Gnp{n_0}{p})})\\
			=	&\sum_{t=1}^{k_1}\sum_{S\subseteq U;\abs{S}=t}\Pof{S\text{ is an independent set in } U}\\
			&\cdot	\prod_{u\in U\setminus S} \left(1-(1-p)^t\right)\\
			=   &\sum_{t=1}^{k_1}\sum_{S\subseteq U;\abs{S}=t} (1-p)^{\binom{t}{2}} \left(1-(1-p)^t\right)^{n_0-t}\\
			&\tag*{(since $\left|U\setminus S\right| = n_0-t$)}\\
			=	&\sum_{t=1}^{k_1} \binom{n}{t} (1-p)^{\binom{t}{2}} \left(1-(1-p)^t\right)^{n_0-t}\\
			\leq&\sum_{t=1}^{k_1} \left(\frac{ne}{t} (1-p)^{\frac{t-1}{2}}\right)^t \exp\left(-(n_0 - t)
			(1-p)^t\right)  
			&\tag*{(since$ 1-x\leq e^{-x}\text{ for all } x$)}\\
			=& \sum_{t=1}^{k_1} \underbrace{\frac{(1-p)^{\binom{t}{2}}}{t^t}}_{\leq 1} n^t \exp\left(t\right) \exp\left(t(1-p)^t\right) \exp\left(-n_0 (1-p)^t\right)\\
			\leq& \sum_{t=1}^{k_1} n^t \exp\left(t+t(1-p)^t\right) \exp\left(-n_0 (1-p)^t\right)\\
			=&\sum_{t=1}^{k_1} \left(n \exp(1+ (1-p)^t) \right)^t \exp(-n_0 (\underbrace{1-p}_{=b^{-1}})^t)\\
			\leq& k_1 \left(n \exp(\underbrace{1+ (1-p)^{k_1}}_{\leq 2}) \right)^{k_1} \exp\left(-n_0 b^{-k_1}\right)
			&\tag*{(since $b>1$)}\\
			=& k_1 (ne^2)^{k_1} \exp\left(-\frac{np}{\log_b(n)^2} 
			\exp\left(-\log(b) (\log_b(np) - \kappa \log_b(\log_b(n)))\right)\right) \\
			=& k_1 (ne^2)^{k_1} \exp\left(-\frac{np}{\log_b(n)^2} \exp\left(-\log(b) \frac{(\log(np) - \kappa \log(\log_b(n)))}{\log(b)}\right)\right)\\
			=& k_1 (ne^2)^{k_1} \exp\left(-\frac{np}{\log_b(n)^2} \exp\left( -(\log(np) +\kappa \log(\log_b(n)))\right)\right)\\
			=& k_1 (ne^2)^{k_1} \exp\left(-\frac{np}{\log_b(n)^2} \frac{\log_b(n)^{\kappa}}{np}\right)\\
			=& k_1 (ne^2)^{k_1} \exp\left(-\log_b(n)^{\kappa-2}\right)\\
			=& \exp\left(\underbrace{\log(k_1)+ k_1 \log(n) + 2 k_1}_{\oOf{\log_b(n)^{\kappa-2}}} -\log_b(n)^{\kappa-2}\right)\\
			=& \exp\left(-(1-  o(1))\log_b(n)^{\kappa-2}\right) &\tag*{(as shown above)}\\
			\leq& \exp\left(- (1\pm o(1)) \log(n)^{\kappa-2} p^{-(\kappa - 2)}\right) 
			&\tag*{(with $\log(b) = \left(1\pm \oOf{1}\right) p$)}\\
			\leq& \exp\left(- (1\pm o(1)) n^{\theta (\kappa - 2) }\right) 
			&\tag*{$\left(\text{since }p\leq n^{-\theta} \text{and} \log(n)^{\kappa-2} \geq 1\right)$}\\ 
			=	& \exp\left(- (1\pm o(1)) n^{\theta (\frac{4}{\theta} - 2) }\right) 
			\tag*{$\left( \text{as }\kappa=\frac{4}{\theta}\right)$}\\
			=	& \exp\left(- (1\pm o(1)) n^{4-2\theta}\right)\\
			\leq& \exp\left(- (1\pm o(1)) n^3\right). 		
			\tag*{$\left(\text{since } \theta \leq \frac{1}{3}\right)$}
	\end{align*}}
	Thus the probability that in every set of at least $n_0$ vertices every maximally independent set is of size at least $k_1$
	is at least $1-\exp\left(- c n^3\right).$
	So in each step before the number of uncolored vertices drops below $n_0$, at least $k_1$ vertices are colored.
	Therefore, the probability that more than $ (1+\epsilon)\frac{n}{\log_b(n)} \geq \frac{n}{k_1} +n_0$ colors are used is at most 
	$\exp\left(- (1+o(1)) n^3\right).$
	
\end{proof}

We now consider colorings of $\Gnp{n}{p}$ using an asymptotically optimal amount of colors for the case $p(n) \leq n^{-\theta}$. 
The following algorithm can be used to obtain a coloring which is an $1+\epsilon$ approximation of an optimal coloring with a 
slightly better probability bound for $p(n) \sim n^{-\theta}$ as in \cite{bollobas_chromatic_1988}.
This will be used later.

\begin{lemma}\label{lem:no_independence_Gnp_small_p}
	Let $n^{-\tau}\leq p \leq n^{-\theta}$ for some arbitrary, 
	but fixed $\tau,\theta \in \left(0,\frac{1}{3}\right)$.
	Then the probability that there exists a vertex set $V_1 \subseteq V$
	of size $\nu(n)=\frac{n}{\log(np)}$ such that $V_1$ does not contain an independent
	set of size $k_0(\nu)\sim 2 \log_b(np) \sim 2 \log_b(np) - 4 \log_b(\log(np))$
	is bounded from above by
	\begin{align*}
		\Pof{\exists S\subseteq [n] : \abs{S} = \nu \land \alpha(\Gnp{n}{p}[S]) < k_0(\nu)} \leq 
		\exp\left(-  \left(\frac{c_2 n^2 p^3}{\log(n)^6}\right)\right).
	\end{align*}
\end{lemma}
\begin{proof}
	Let $ \nu = \frac{n}{\log(np)}.$ We have $k_0(\nu) \sim 2 \log_b(np) \sim 2 \log_b(np) - 4 \log_b(\log(np))$.
	Now,
	\begin{align}
		&\Pof{\exists S\subseteq [n] : \abs{S} = \nu \land \alpha(\Gnp{n}{p}[S]) < k_0(\nu)}\notag\\
		\leq&\binom{n}{\nu} \exp\left(- \frac{c_1 \nu^2 p^3}{\log(\nu)^4}\right) \tag{by \autoref{thm:independence_all_p_aas} for some constant $c_1>0$}\\
		\leq&\exp\left(- \nu \frac{c_1 \nu p^3}{\log(\nu)^4} +\nu\log(n)\right)\notag\\
		=&\exp\left(- \frac{n}{\log(np)} \frac{c_1 n p^3}{\log(np)(\log(n)-\log(\log(np)))^4} +\frac{n}{\log(np)}\log(n)\right)\notag\\
		\leq&\exp\left(-  \left(\frac{c_2 n^2 p^3}{\log(n)^6}\right)\right), \label{ineq:Gn_p_small_P_prf_alg_1}\tag{a.3}
	\end{align}
	and the last inequality holds for a constant $c_2>0$ and sufficiently large $n,$
	because $\log(np) \leq \log(n n^{-\theta}) \leq  (1-\theta)\log(n).$
\end{proof}

We proceed to the proof of \autoref{thm:chromatic_small_p_Gnp} using \autoref{alg:chrom_small_p_Gnp} and \autoref{lem:no_independence_Gnp_small_p}.

\begin{proof}[Proof of \autoref{thm:chromatic_small_p_Gnp}]
	By \autoref{lem:no_independence_Gnp_small_p}  the probability that \autoref{alg:chrom_small_p_Gnp} fails to find an 
	independent set of size $k_1:= 2\log_b(np) - \log_b(\log_b(np))$ as long as at least $\nu= \frac{n}{\log(n)}$
	vertices remain is at most $\exp\left(-  \left(\frac{c_2 n^2 p^3}{\log(n)^6}\right)\right).$
	Hence the number of colors used in the while-loop of \autoref{alg:chrom_small_p_Gnp} is at most 
	$\frac{n}{k_1} \leq (1+\frac{\epsilon}{2}) \frac{n\log(b)}{\log(np)}$. 
	Furthermore, since the remaining $\nu$ vertices induce a $\Gnp{\nu}{p}$, 
	the probability that \autoref{alg:chromatic_greedy} uses more than 
	\begin{align}
		&(1+\epsilon)\frac{\nu \log(b) }{\log(\nu)} \notag\\
		=	&(1+\epsilon)\frac{\nu \log(b) }{\log(n)-\log\log(np)}\notag\\
		=	&(1+\epsilon)\frac{\nu \log(b) }{(1-\oOf{1})\log(n)}\notag\\
		=	&\frac{1+\epsilon}{1-\oOf{1}} \cdot \frac{n \log(b) }{\log(np)\log(n)}\notag\\
		\leq&\frac{1+\epsilon}{1-\oOf{1}} \cdot \frac{n \log(b) }{\log(np)^2}\notag\\
		=	&\frac{\epsilon}{2}\frac{n \log(b)}{\log(np)} \cdot 
		\underbrace{\frac{2}{\epsilon} \frac{1+\epsilon}{(1-o(1))\log(np)}}_{<1 \text{ for sufficiently large } n}\notag\\
		\leq&\frac{\epsilon}{2}\frac{n \log(b)}{\log(np)} \notag
	\end{align}
	colors to color the remaining $\nu$ vertices can be estimated by \autoref{thm:chromatic_greedy} to be at most
	\begin{align}
		\exp\left(- (1+o(1)) \nu^3 \right) = \exp\left(-\frac{(1+o(1))n^3}{{\log(n p)}^3}\right) \leq  \exp\left(-\left(\frac{(1+o(1))n^3}{\log(n)^3}\right)\right).\label{ineq:Gn_p_small_P_prf_alg_2} \tag{a.4}
	\end{align}
	
	Thus the probability that \autoref{alg:chrom_small_p_Gnp} uses more  than 
	$(1+ \epsilon) \frac{n\log(b)}{2\log(np)}$ colors is at most 
	\begin{align*}
		&\Pof{\text{\autoref{alg:chrom_small_p_Gnp} uses more than} (1+ \epsilon) \frac{n\log(b)}{2\log(np)} \text{ colors} }\\
		\leq &\Pof{\text{The while-loop of \autoref{alg:chrom_small_p_Gnp} colors less than $\nu$ vertices}}\\
		&+ \Pof{\text{\autoref{alg:chromatic_greedy} uses more than } \frac{\epsilon}{2}\frac{n\log(b)}{\log(np)} \text{colors}}\\
		\leq &\exp\left(-  \left(\frac{c_2 n^2 p^3}{\log(n)^6}\right)\right)+
		\exp\left(-\left(\frac{(1+o(1))n^3}{\log(n)^3}\right)\right)\\
		&\tag*{(using \autoref{lem:no_independence_Gnp_small_p} and \ref{ineq:Gn_p_small_P_prf_alg_2})}\\
		\leq &\exp\left(-\left(\frac{c_3 n^2 p^3}{\log(n)^6}\right)\right),		
	\end{align*}
	for some constant $c_3 >0.$
	The last inequality holds due to the fact that \newline
	$\exp\left(-\left(\frac{(1+o(1))n^3}{\log(n)^3}\right)\right) \in \oOf{\exp\left(-  \left(\frac{c_2 n^2 p^3}{\log(n)^6}\right)\right)}$.\\
\end{proof}

Now we are able to prove \autoref{lem:chromatic_small_p_small_components}.
\begin{proof}[Proof of \autoref{lem:chromatic_small_p_small_components}]
	Let $\epsilon>0$ and $k= (1+\frac{\epsilon}{2}) \frac{np}{2\log(np)}$.
	Let further $n$ be sufficiently large. Then,
	\begin{align*}
		\Eof{\chromaticOf{\Gnp{n}{p}}}
		&= \sum_{i=1}^n i \cdot \Pof{\chromaticOf{\Gnp{n}{p}} =i} \\ 
		&= \sum_{i=1}^k i \cdot \Pof{\chromaticOf{\Gnp{n}{p}} =i} \sum_{i=k+1}^n i \cdot \Pof{\chromaticOf{\Gnp{n}{p}} =i}\\ 
		&\leq \sum_{i=1}^{k} i \cdot \Pof{\chromaticOf{\Gnp{n}{p}} =i} + n \Pof{\chromaticOf{\Gnp{n}{p}} > k} \\
		&\leq (1+\frac{\epsilon}{2}) \frac{np}{2\log(np)} + n \exp\left(-\left(\frac{c n^2 p^3}{\log(n)^5}\right)\right)
		&\tag*{(by \autoref{thm:chromatic_small_p_Gnp} with $\epsilon^\prime := epsilon/2$)}\\
		&=  (1+\epsilon) \frac{np}{2\log(np)}.
	\end{align*}
	For the last inequality we use 
	\begin{align*}
		n \exp\left(-\left(\frac{c n^2 p^3}{\log(n)^5}\right)\right) \leq \frac{np}{2\log(np)}.
	\end{align*}
	This holds for sufficiently large $n,$ because
	\begin{align*}
		\lim\limits_{n\rightarrow\infty} n \exp\left(-\left(\frac{c n^2 p^3}{\log(n)^5}\right)\right) = 0 
		\text{ and }
		\lim\limits_{n\rightarrow\infty} \frac{np}{2\log(np)}=\infty.
	\end{align*}
	
\end{proof}

\end{document}